\documentclass[12pt]{amsart}
\usepackage{graphicx,amssymb,latexsym, amsmath,amsthm}

\usepackage[english]{babel}
\usepackage[latin1]{inputenc}
\usepackage{xy}
\usepackage{graphicx}
\usepackage{verbatim}
\usepackage{enumerate}
\usepackage{url}
\usepackage{subfigure}
\usepackage{float}
\usepackage{multicol}
\usepackage{mathtools}

\newcommand{\IS}{\overset{\mathbb{I}}{=}}
\newcommand{\TS}{\overset{\mathbb{F}}{=}}
\newcommand{\examend}{\hfill \mbox{$\diamondsuit$}}

\newcommand{\NS}{N_0}
\newcommand{\CC}[1]{M_\bullet(#1)}

\newcommand{\Ref}[1]{(\ref{#1})}
\newcommand{\real}{\mathbb{R}}

\newcommand{\XX}{\mathbf{X}}
\newcommand{\YY}{\mathbf{Y}}
\newcommand{\xx}{{\mathbf{x}}}
\newcommand{\yy}{{\mathbf{y}}}

\newcommand{\ww}{{\mathbf{w}}}

\newcommand{\arr}{\rightarrow}
\newcommand{\dd}{\,|\,}
\newcommand{\ECS}{{\boldsymbol{\emptyset}}}

\newtheorem{lemma}{Lemma}[section]

\newtheorem{prop}[lemma]{Proposition}
\newtheorem{thm}[lemma]{Theorem}
\newtheorem{cor}[lemma]{Corollary}
\theoremstyle{definition}
\newtheorem{Def}[lemma]{Definition}
\newtheorem{exam}[lemma]{Example}
\newtheorem{exams}[lemma]{Examples}
\theoremstyle{remark}
\newtheorem{rem}[lemma]{Remark}
\newtheorem{rems}[lemma]{Remarks}

\newcommand{\conc}{\diamond}
\newcommand{\ovee}{{\overline{\vee}}}

\newcommand{\sm}{\smallsetminus}

\newcommand{\mc}[1]{\mathcal{#1}}
\newcommand{\A}{\mc{A}}
\newcommand{\Ee}{\mc{E}}

\newcommand{\Mb}{\mathbf{M}}

\newcommand{\is}[1]{\mathbf{#1}}

\newcommand{\Nn}{\mathcal{N}}

\newcommand{\NN}{\mathbf{N}}

\newcommand{\rset}[2]{\left\lbrace #1 \dd #2\right\rbrace}

\newcommand{\set}[2]{\rset{#1}{#2}}
\newcommand{\tset}[2]{\big\lbrace #1\,\big|\;#2\big\rbrace}
\newcommand{\sset}[1]{\left\lbrace #1\right\rbrace}
\newcommand{\tsset}[1]{\big\lbrace #1\big\rbrace}

\newcommand{\R}{\mathbb{R}}

\newcommand{\Net}{\mathfrak{N}}
\newcommand{\F}{\mathcal{F}}

\newcommand{\f}{\mathbf{f}}

\newcommand{\Sw}{\widehat{\is{S}}}
\newcommand{\sw}{\widehat{S}}

\newcommand{\In}{\mathbb{I}}
\newcommand{\InH}{\widehat{\In}}

\newcommand{\Tr}{\mathbb{F}}
\newcommand{\TrH}{\widehat{\Tr}}

\renewcommand{\aa}{\alpha}

\newcommand{\FAN}{\mathbf{N}}
\newcommand{\FANw}{\widehat{\FAN}}
\newcommand{\DDN}{\mathcal{D}(\Net,\In,\Tr)}
\newcommand{\DD}{\mathcal{D}(\FAN)}

\newcommand{\HDD}{\widehat{\mathcal{D}}(\FAN)}

\newcommand{\EvOp}{\widehat{\Phi}}
\newcommand{\EvOpc}{\widehat{\Phi}^c}

\newcommand{\TT}{\mathbf{T}}
\newcommand{\Ss}{\mathbf{S}}

\newcommand{\s}{\mathrm{c}}

\title[Asynchronous Networks]{Asynchronous Networks: Modularization of Dynamics Theorem} 
\author{Christian Bick}
\thanks{Research supported in part by NSF Grant DMS-1265253 \& Marie Curie IEF Fellowship (Project 626111)}
\address{Department of Mathematics, Rice University MS-136, 6100 Main St., Houston, 
TX~77005, USA}
\curraddr{Department of Mathematics, University of Exeter, Exeter EX4 4QF, UK}
\email{c.bick@exeter.ac.uk}

\author{Michael Field}
\thanks{Research supported in part by NSF Grant DMS-1265253 \& Marie Curie IIF Fellowship (Project 627590)}
\address{Department of Mathematics, Rice University MS-136, 6100 Main St., Houston, 
TX~77005, USA}
\curraddr{Department of Mathematics, Imperial College, SW7 2AZ, UK}
\email{mikefield@gmail.com, Michael.Field@imperial.ac.uk}
\date{\today}

\begin{document}

\begin{abstract}
Building on the first part of this paper, we develop the theory of
functional asynchronous networks. We show that a large class of 
functional asynchronous networks can be (uniquely) represented as
feedforward networks connecting events or dynamical modules. For these
networks we can give a complete description of the network function in terms of the
function of the events comprising the network: the Modularization of Dynamics Theorem.
We give examples to illustrate the main results. 
\end{abstract}


\newcommand{\imagescaling}{0.9}

\maketitle

\let\oldtocsection=\tocsection
\let\oldtocsubsection=\tocsubsection
\let\oldtocsubsubsection=\tocsubsubsection

\renewcommand{\tocsection}[2]{\hspace{-\normalparindent}\oldtocsection{#1}{#2}}
\renewcommand{\tocsubsection}[2]{\hspace{0em}\oldtocsubsection{#1}{#2}}
\renewcommand{\tocsubsubsection}[2]{\hspace{1em}\oldtocsubsubsection{#1}{#2}}

\tableofcontents


\section{Introduction}

In this work we the develop of the theory of \emph{functional asynchronous networks}. Previously, in~\cite{BF0}, we
gave the general definition and formalism for an asynchronous network, together with some examples and results about products. 
We assume some familiarity
with~\cite{BF0} in what follows (most specifically, sections 2, 4 and 6 of~\cite{BF0}). 

The term `functional network' has been used previously. For example, classes of functional networks,
which have relations with control theory, have previously been
considered in a neuroscientific context~\cite{Bick2009a, RabVar2011, RabReview, Bick2014a}, and 
in homological studies of brain function~\cite{PETCH,SHP}.

In our context, a functional asynchronous network will be a network with
a prescribed set $\In$ of initializations and terminations $\Tr$. If the network phase space is $\Mb$, then $\In, \Tr$ will be closed disjoint subsets of
$\Mb$.  Roughly speaking, the \emph{function} of the network will be to get from any point $\XX \in \In$ to a point in $\Tr$ in finite time. 

Our main result will be to show that the function of a large class of functional asynchronous networks can be understood in terms
of the functions of the events that comprise the network. 

\begin{figure}[h]
\centering
\includegraphics[width=\textwidth]{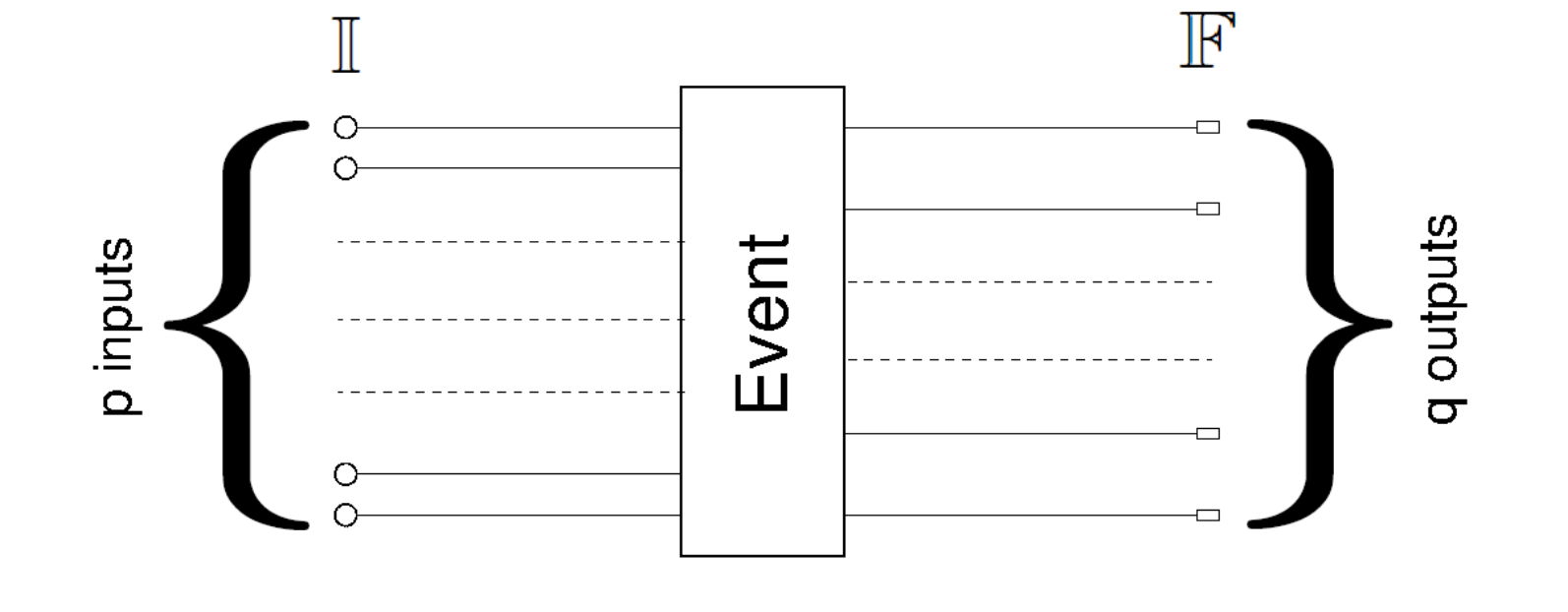}
\caption{An Event or Dynamical Module}
\label{prim}
\end{figure}

Referring to figure~\ref{prim}, we regard an event as a `dynamical module' that accepts a number of inputs 
and has a number of outputs (where each input and output corresponds to the state of a node).  In the figure 
we allow different numbers of inputs and outputs (see section \cite[\S 5]{BF0}) but in the present work we make the simplifying assumption 
that the event has the same number of inputs and outputs (this is not required for our main result).  Now imagine that a functional asynchronous network is built by
coupling together a finite set of dynamical modules -- see figure~\ref{prim1} for a nine node network built using eight events or dynamical modules.
\begin{figure}[h]
\centering
\includegraphics[width=\textwidth]{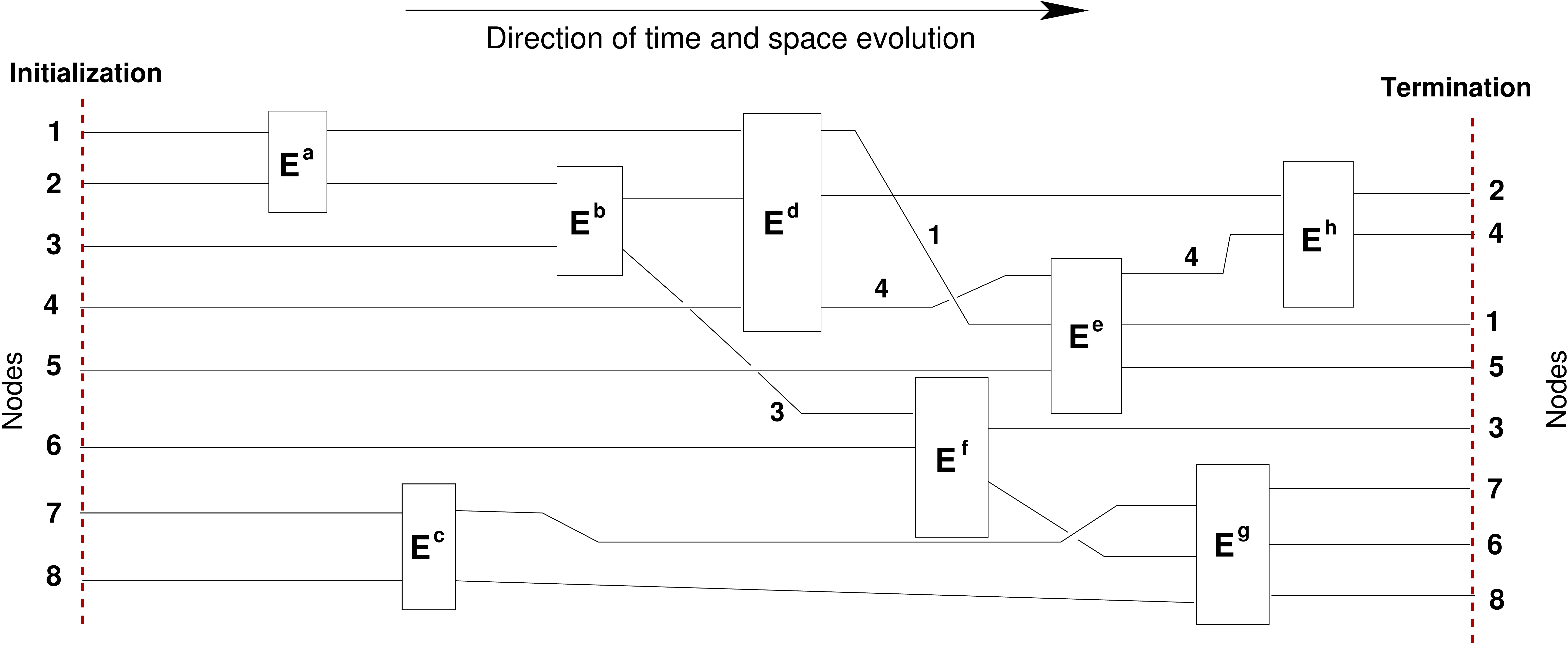}
\caption{A feedforward network built from events}
\label{prim1}
\end{figure}

Our first main result identifies a large class of functional asynchronous networks which have a unique representation as feedforward networks built from events of the type described above. 
Our second result shows that for these networks, the
function of the original network can be completely described in terms of the functions of the events comprising the feedforward network.
We refer to the two results as the \emph{Modularization of Dynamics Theorem}.

We conclude by describing the contents of the paper in more detail.  In section~\ref{secfun} we give the formal definitions of a functional asynchronous network,
initialization and termination sets, and network function. We also discuss the phenomenon of \emph{dynamical deadlocks}. In section~\ref{RWRF}, we give the key definitions of
geometric, weakly regular and regular asynchronous networks and construct the evolution operator that allows for generalized initialization in space and time.  We define functional asynchronous networks of simple type and show how every weakly regular asynchronous network has an associated weakly regular network of simple type with the same network function. 
We conclude with some comments about hidden deadlocks. In section~\ref{join:sec}, we define the operations of
amalgamation and concatenation for families of functional asynchronous networks of simple type that share the same node set.
In section~\ref{factorsec} we state and prove our main results and give some illustrative examples. We conclude in section~\ref{conclu} with
comments and a brief discussion of some outstanding problems.

\section{Functional asynchronous networks}
\label{secfun}
In section~5.2~\cite{BF0},  we introduced the idea of a functional asynchronous network
in the setting of a transport network.  Functional asynchronous networks will be central 
to the formulation and proof of the modularization of dynamics theorem. In this section we give basic definitions
and properties as well as examples that illustrate the phenomenom of a dynamical deadlock.

We continue with the notational conventions of~\cite{BF0}. In particular, $\Net=(\mathcal{N},\A,\mathcal{F},\Ee)$ 
will always denote a proper asynchronous network with node set $\Nn = \{\NS,N_1,\dotsc,N_k\}$ and associated semiflow 
\[
\Phi=(\Phi_1,\dotsc,\Phi_k):\prod_{i\in\is{k}}M_i \times \real_+=\Mb\times\real_+ \arr \Mb.
\]

\subsection{Initialization, termination and network function}

\begin{Def}\label{init-term}
Closed subsets $\In,\Tr$ of $\Mb$ are \emph{initialization and termination sets} for $\Net$
if 
\begin{enumerate}
\item[(P1)] there are closed disjoint subsets $\In_i,\Tr_i$ of $M_i$, $i \in \is{k}$,  
such that
\[ 
\In = \prod_{i=1}^k\In_i, \qquad
\Tr = \prod_{i=1}^k\Tr_i 
\]
\item[(P2)] If $\XX\in \In$, then  for each $i \in \is{k}$ there exists $t_i(\XX) \ge 0$ such that $\Phi_i(\XX,t) \in \In_i$
if and only if $t \in [0,t_i(\XX)]$.
\end{enumerate}
\end{Def}
\begin{lemma}
\label{spropinit}
Let  $\In,\Tr$ be initialization and termination sets for $\Net$. Then
\begin{enumerate}
\item[{\rm (P3)}] $\In$ contains no compact $\Phi$-invariant sets.
\item[{\rm (P4)}] If $\XX \in \In$ and there exists $t \ge 0$ such that $\Phi_i(\XX,t) \in \Tr_i$, then $t > t_i(\XX)$. 
\end{enumerate}
If {\rm (P4)} applies and we let 
$S_i = S_i(\XX) =\inf_{t \ge 0}\sset{t \dd \Phi_i(\XX,t) \in \Tr_i}$ denote the \emph{transit time} from~$\XX$ to~$\Tr_i$,
then $\Phi_i(\XX,S_i) \in \Tr_i$.
\end{lemma}
\proof Obviously (P2) $\Longrightarrow (P3)$; (P4) follows since $\In_i,\Tr_i$ are closed disjoint sets. \qed

\begin{rems}
\label{P:rems}
(1) We do not require that for \emph{every} $\XX \in \In$, $i \in \is{k}$, 
there is a transit time $S_i$ for which $\Phi_i(\XX,S_i) \in \Tr_i$. Moreover, if
$\Phi_i(\XX,S_i) \in \Tr_i$, it may or may not be the case that $\Phi_i(\XX,t) \in \Tr_i$, for $t > S_i$. 
The transit time $S_i$ is the time that the state of node $N_i$ first  enters $\Tr_i$. 
For some examples, it is natural to have $\Phi_i(\XX,t) = \Phi_i(\XX,S_i)$ 
for all $t \ge S_i$ (so $\Phi_i(\XX,S_i)$ really is a terminal 
state for the \emph{semiflow} $\Phi$, not just for the function of
transitioning from points in $\In$ to $\Tr_i$, $i \in \is{k}$). 
In other situations, states may continue to evolve under~$\Phi$. \\
(2) Later it will sometimes be useful to allow $\In_i = \Tr_i$ in definition~\ref{init-term}. In this case, 
we take the transit time $S_i$ to be zero. \\
(3) It may be the case that a phase space $M_i$ has 
boundary~$\partial M_i$ and 
$\partial M_i \supset \In_i \cup \Tr_i$ (for example, 
the passing loop example of \cite[\S 5.1]{BF0}). If so, it is
natural to assume that $\Phi_i(\XX,t) = \Phi_i(\XX,S_i)$, $t \ge S_i$.\\
(4) Condition (P2) implies trajectories do not re-enter $\In$: if $t > t_i(X)$, then
$\Phi_i(\XX,t) \notin \In_i$ and so if $t > \max_i t_i(X)$, 
$\Phi_i(\XX,t) \notin \In_i$, all $i \in\is{k}$.
\end{rems}

\begin{Def}
\label{def:connect}
Let $\In,\Tr$ be initialization and termination sets for~$\Net$. A 
point $\XX \in \In$ is \emph{$\Phi$-connected to~$\Tr$} if there 
exists $\YY = (\yy_1,\dotsc,\yy_k) \in \Tr$ and transit times
$\is{S} = \is{S}(\XX) = (S_1,\dotsc,S_k) \in \real_+^k$ such that
\[
\yy_i = \Phi_i(\XX,S_i), \; i \in \is{k}.
\]
\end{Def}

\begin{rems}
\label{rem:Term}
(1) With the notation of definition~\ref{def:connect}, we say 
\emph{$\XX$ is $\Phi$-connected to $\YY$}. If~$\XX$ is $\Phi$-connected 
to~$\YY$, then $S_i=S_i(\XX)$ is always the minimal transit time 
from~$\XX$ to $\Tr_i$. Setting $\is{S}(\XX) = (S_1,\dotsc,S_k)$ and 
abusing notation, we often write 
$\YY=(\yy_1,\dotsc,\yy_k)= \Phi(\XX,\is{S}(\XX))$. We refer to~$\yy_i$ 
as the \emph{terminal state} of~$N_i$, $i \in \is{k}$.\\
(2) If $\XX \in \In$ is $\Phi$-connected to~$\Tr$, this does not imply 
that the $\Phi$-trajectory through~$\XX$ meets~$\Tr$. Even 
if there exists $s > 0$ such that $\Phi(\XX, s)\in\Tr$, then 
$s$ and $\Phi(\XX, s)$ may not give any of the transit times and 
terminal states. That is, we may have 
$S_i<s$ and
$\Phi_i(\XX, s)\neq \yy_i$  for all $i \in \is{k}$.
However, if 
$\Phi_i(\XX,t) = \Phi_i(\XX,S_i)$ for all $t\geq S_i$ 
then $\XX\in\In$ is $\Phi$-connected 
to~$\Tr$ if and only if the $\Phi$-trajectory through~$\XX$ 
meets~$\Tr$. 
\end{rems}

Let $\In,\Tr$ be initialization and termination sets for~$\Net$ and 
set
\[
\DDN = \set{\XX \in \In}{\XX\;\text{is $\Phi$-connected to } \Tr}.
\]
\begin{Def}
(Notation and assumptions as above.)
The \emph{transition function} $G_0: \DDN \subset \In \arr \Tr$ is defined 
by
\[
G_0(\XX) = \Phi(\XX,\is{S}(\XX)),\;\XX \in \DDN.
\]
\end{Def}

\begin{Def}
\label{def:fan}
Let $\In,\Tr$ be initialization and termination sets for $\Net$.
The triple $\FAN = (\Net,\In,\Tr)$ is a \emph{functional asynchronous 
network (FAN)}.  The  \emph{network function} is transition from points in $\In$ to $\Tr$ 
and is represented by the transition function $G_0:  \DD \subset \In \arr \Tr$. The
network function is \emph{achieved} if $ \DD = \In$ -- that is, if every point 
in~$\In$ is $\Phi$-connected to~$\Tr$ and $G_0: \In\arr\Tr$.
\end{Def}
\begin{rems}
(1) In general, $\DD$  may be a proper subset of
$\In$. Achieving the network function may be part of the process of network design 
and involve a mix of dynamics and logical conditions (see below).\\
(2) In the sequel we regard the achievement of network function as synonymous with the transition
function $G_0$ having domain $\In$. 
\end{rems}

\subsection{Deadlocks}
In this section we address one of the reasons for the failure of a FAN to achieve its function: the presence of deadlocks. 

\begin{Def}
A FAN $\FAN=(\Net,\In,\Tr)$ 
has a \emph{dynamical deadlock} if there is
a nonempty subset $A \subset \Mb$ such that
\begin{enumerate}
\item $A$ is compact and semiflow invariant: $\Phi^t(A) = A$, $t \ge 0$.
\item $A \cap (\In \cup \Tr) = \emptyset$.
\item There is a nonempty subset $K$ of $\In\sm\DD$ such that every 
trajectory through a point of $K$ enters $A$ within finite time.
\end{enumerate}
We refer to~$A$ as \emph{deadlock sink}. If~$K$ contains an open 
set, $A$ is a \emph{topological deadlock sink}, and if~$K$ has 
nonzero Lebesgue measure, $A$ is an \emph{observable deadlock sink}.
\end{Def}

\begin{exam}
\label{deadex}
In the passing loop example of \cite[\S 5.1]{BF0}, a required  condition for exiting the passing loop was that
two coupled phase oscillators were phase synchronized to within $\varepsilon\in (0,0.5)$. Assuming identical frequencies, 
phase oscillator dynamics is given by 
\[
\theta_1' = \omega + k \sin 2\pi (\theta_2 - \theta_1),\quad \theta_2' = \omega + k \sin 2\pi (\theta_1 - \theta_2),
\]
where $k> 0$. If $|\theta_2(0) - \theta_1(0)| = 0.5$, 
then $|\theta_2(t) - \theta_1(t)| = 0.5$ for all $t \ge 0$ and so there is a
deadlock with deadlock sink 
\[
A = \{((0,0),(\theta,\theta+0.5)) \dd \theta \in \mathbb{T}\}\subset [-a,b]^2\times \mathbb{T}^2. 
\]
(For this example, 
$\In_1 = \Tr_2 = \sset{-a} \times \mathbb{T}$, $\Tr_1 = \In_2 = \sset{b} \times \mathbb{T}$ and the deadlock will not be 
observable.) 
\examend
\end{exam}

The next lemma shows that dynamical deadlocks cannot occur in networks 
governed by a single set of differential equations.

\begin{lemma}
The  FAN $\FAN=(\Net,\In,\Tr)$ has no dynamical deadlocks if
$\mathcal{E}|\Mb \smallsetminus (\In\cup \Tr)$ is 
constant.
\end{lemma}

\begin{proof} Suppose that $A$ is a deadlock sink for $\FAN$. 
Let~$\mathbf{F}$ be the the vector field on~$\Mb$ determined by
$\mathcal{E}|\Mb \smallsetminus (\In\cup\Tr)$. 
Clearly, $A$ is invariant by the flow of~$\mathbf{F}$.  Since~$A$ 
has an open neighbourhood in 
$\Mb \smallsetminus (\In\cup \Tr)$, 
no trajectory starting in $\In$ (or $\Mb\smallsetminus A$) 
can enter $A$ in finite time, contradicting our assumption 
that~$A$ is a deadlock sink for $\FAN$.
\end{proof}

\begin{Def}
Let  $A$ be a deadlock sink for the FAN $(\Net,\In,\Tr)$.
\begin{enumerate}
\item $A$ is a \emph{deadlock} if $A$ consists of a single point $(a_1,\dotsc,a_k)$. The deadlock is total if $a_i \notin \Tr_i$ for all
$i \in \is{k}$, and \emph{partial} if there exists $i \in \is{k}$ such that $a_i \in \Tr_i$.
\item $A$ is a \emph{livelock} if $A$ is a periodic orbit.
\end{enumerate}
\end{Def}

\begin{exams}\label{ex:FourWayStop}
(1) The deadlock given in example~\ref{deadex} is a livelock  -- an antiphase (periodic) solution of
the phase oscillator pair.\\
(2) Deadlocks can occur because of faulty logic. 
For example, if we have a 4-way stop sign at a cross roads with the following traffic rules.
\begin{itemize}
\item[(a)] All traffic arriving at the stop sign stops.
\item[(b)] No left or right turns allowed -- when a vehicle restarts it proceeds in same direction whence it came.
\item[(c)] Priority is determined by order of arrival -- first to arrive (stop), first to leave (restart) -- unless more than one vehicle
arrives at the stop sign at the same time, in which case priority is given to the car on the right (there is no issue if two
vehicles arrive at the same time from opposite directions).
\end{itemize}
A total deadlock occurs if four cars arrive at the stop sign at the same time. There is no easy way to vary the logic to resolve the deadlock.
Provided the traffic is light, the deadlock may be regarded as
acceptable\footnote{An example of the \emph{ostrich algorithm} from computer science.}, otherwise it may be preferable
(a) to have a two way stop sign, or  (b) use traffic lights. Although the deadlock is, in principle, not observable, in practice
it is: drivers are only able to approximately judge the time of arrival of vehicles on their right. 
\examend
\end{exams}

The presence of deadlocks can prevent a functional network from completing
its function. We briefly discuss some examples of deadlocks in real-world networks 
and how they can be resolved.

\subsubsection{Resource allocation}
In computer science and distributed systems, deadlocks are typically found in
problems with resource allocation. In network terminology, two (or more)
nodes trying to connect to a third node $N_T$ which only allows one connection.  
For example, in threaded computation, data corruption can occur if
two threads attempt simultaneous writes to the same memory.  The resolution involves
(a) the connected node $N_S$ having a \emph{lock} on the target node $N_T$ until the process requiring the connection is finished;
(b) a protocol for how to handle the situation when two nodes simultaneously request connection to
$N_T$, (c) prioritisation of connection requests. If the node $N_S$ does not ever get disconnected from
$N_T$, then there will be at least a partial deadlock and if low priority nodes are connected then
the system may run slowly. If the time of connection is small then 
attempts at simultaneous connection may be rare: deadlocks are often very hard to find in complex networks.
In terms of event driven dynamics and asynchronous networks,
(a,b) are handled by a correctly written event structure which may use local time. 
After the connection between  $N_S$ and $N_T$ is released, a new connection will be made either
randomly or consistent with a prioritisation list (or both!).   In large complex networks,
it may be extremely hard to organise the structure so that there are no deadlocks.

\section{Regularity conditions on a FAN}
\label{RWRF}
Our aim in the next two sections is show how we can express the dynamics of a FAN
in terms of the dynamics of subnetworks which are also FANs. 
In the present section, we shift our focus from the absolute 
definition of a FAN, as given in definition~\ref{def:fan},
to a more relative definition where we impose geometric and structural conditions on a (sub)network that give dynamics
that is closely related to the dynamics of the containing global network. As part of this process,  
we will eventually need to relax our assumption that all nodes, even uncoupled nodes, 
are started at exactly the same time (that is, at time $t = 0$).

\subsection{Geometric FANs}

We start with some notational conventions and assumptions that we maintain throughout this section. 
If $(\Net,\In,\Tr)$ is a FAN, then
$\Net = (\mathcal{N},\A,\mathcal{F},\mathcal{E})$ will be proper
asynchronous network with~$k$ nodes, network vector field $\is{F} = (F_1,\dotsc,F_k)$, and well defined 
semiflow~$\Phi=(\Phi_1,\dotsc,\Phi_k)$.  Furthermore, we suppose that $\ECS\in\A$, set
$\mathbf{f}^{\ECS} = \mathbf{Z} = (Z_1,\dotsc,Z_k)$ and assume 
that~$\mathbf{Z}$ determines a smooth flow $\Psi^t = (\psi_1^t,\dotsc,\psi_k^t)$
on~$\Mb$ (automatic if~$\Mb$ is compact).
Since $\mathbf{Z}$ is given by the empty connection structure, 
$\psi_i^t$ determines a flow on $M_i$ for all $i\in \is{k}$. Let $\Mb^\sigma = \prod_{i\in\is{k}}M_i^\sigma$, $\sigma\in \{-,+,0\}$.

\begin{Def}
\label{def:geom}
The FAN $\FAN =(\Net,\In,\Tr) $ is \emph{geometric}
if
\begin{enumerate}
\item[(G)]  For all $i \in \is{k}$, $\In_i, \Tr_i$ are disjoint nonempty closed hypersurfaces in $M_i$ that separate
$M_i$ into nonempty closed connected regions $M_i^-, M_i^0, M_i^+$ with smooth boundary satisfying
\begin{itemize}
\item[(a)] $\partial M_i^- = M_i^-\cap M_i^0 =  \In_i$,
\item[(b)] $\partial M_i^+ = M_i^+\cap M_i^0 = \Tr_i$.
\end{itemize}
\item[(T)] For $i \in \is{k}$, $F_i = Z_i$ on an open neighbourhood of $\In_i\cup\Tr_i$ and is transverse to $\In_i\cup \Tr_i$;
inward pointing for $M_i^0$ along $\In_i$,
outward pointing for $M_i^0$ along $\Tr_i$.
\item[(F)] For each $\XX \in \In$, and $i \in \is{k}$,
there exists a unique smallest $S_i(\XX) \in\real_+$ such that
$\Phi_i(\XX,S_i(\XX))\in \Tr_i$.
\end{enumerate}
\end{Def}
See the following remarks for the geometric implication of these conditions,
and note that the labels \emph{G, T} and \emph{F} refer respectively to
\emph{Geometry, Transversality} and \emph{Function}.

\begin{rems}
\label{if}
(1) Condition (G) implies condition (P1); conditions (G,T) imply condition (P2) (with $t_i(\XX) = 0$,
for all $i \in \is{k}$ and $\XX \in \In$). Condition (P3) follows trivially from (G,T).\\
(2) The geometric conditions on the regions  $M_i^-, M_i^0, M_i^+$ are strong and entail that $\In_i, \Tr_i$ each separate
$M_i$ into two connected components.
For periodic problems, such as daily factory inventory
oscillations or biological rhythms,
$\In_i$, $\Tr_i$ may not disconnect $M_i$ (for example, if $M_i = \mathbb{T}^n$). This would imply that $M_i^- = M_i^+$ 
and allow for trajectories to leave and renter $M_i^0$ or $M_i\smallsetminus M_i^0$. In this regard, we could also allow
$M_i^\pm$ to be empty. \\
(3) Condition~(T) implies that $Z_i | (\In_i\cup\Tr_i)$ is non-vanishing, all $i \in \is{k}$.
\end{rems}
\begin{lemma}
If the FAN $(\Net,\In,\Tr)$ is geometric, then it has a
well-defined transition function $G_0:\In\arr \Tr$ and timing function
$\Ss:\In \arr \real^k_+$. 
\end{lemma}
\begin{proof} Immediate from remarks~\ref{if}(1) and condition (F). \end{proof}

\subsection{Generalizing network function} 
Consider the geometric FAN $\FAN$ shown in figure~\ref{prim3}. Suppose that we initialize at $\is{W} = (\ww_1,\dotsc,\ww_k) \in \Mb^-$. Assume that the state of node $N_i$
enters $\In_i$ at time $T_i \ge 0$: $\Phi_i(\is{W},T_i)=\xx_i \in \In_i$, $i \in \is{k}$ ($T_i = 0$ iff $\ww_i \in \partial M_i^-$). 
Note that if $\ww_i$ is sufficiently close to $\In_i$, all $i \in\is{k}$, then condition (T) implies that each $\Phi_i$ trajectory \emph{will} meet $\In_i$.
\begin{figure}[h]
\centering
\includegraphics[width=\textwidth]{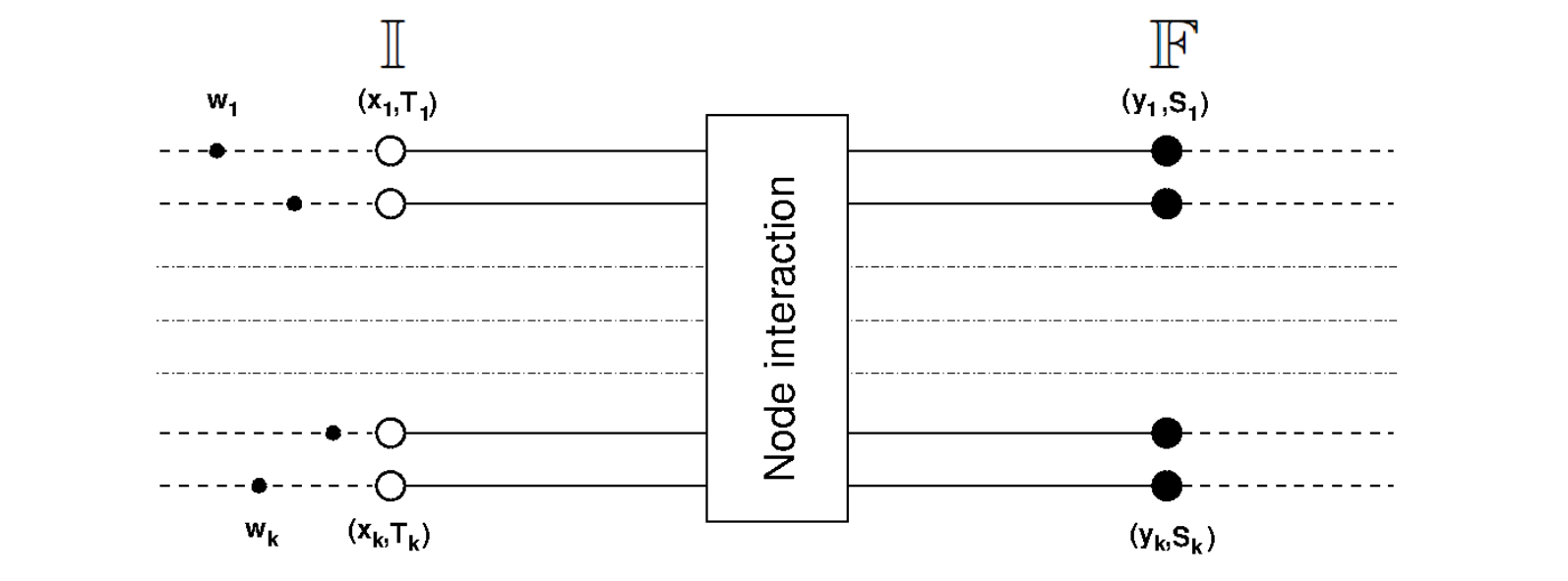}
\caption{Three different ways of viewing a geometric FAN}
\label{prim3}
\end{figure}

If we have $T_1 = \dotsc = T_k = T$, then it follows from condition (F) that there exist times $S_i > T$ such that
$\Phi_i(\is{W},S_i+T) = \Phi_i(\is{X},S_i)=\yy_i \in \Tr_i$, all $i \in \is{k}$. If we do not assume that that the $T_i$ are all equal, then
the $\Phi_i (\is{W},t)$-trajectory may not even meet $\Tr_i$ 
and, even if it does, the time and point of intersection may be different from $S_i+T_i$ and $\yy_i$ 
(we give an example below). 

It is natural to find conditions on $\FAN$ that allow for the initialization of the components $\xx_i$ of $\XX\in\In$ to occur at different times and still achieve network function.  
For example, the terminal times and states should be the same if we either initialize at $\is{W}$ and $t = 0$ or we initialize each component $\xx_i\in\In_i$ at time $T_i$.
The main issue in carrying out this program is that the state of a node $N_j$ not in $M_j^0$ may influence the evolution of a node $N_i$ with state
in $M_i^0$.  For example, the condition for a constraint $\NS\arr N_i$ may depend on $\xx_j \in M_j^- \cup M_j^+$.

Exactly the same problem may occur if we attempt to redefine the semiflow $\Phi$ on $\Net$ by stopping nodes when they reach 
their termination set. That is, if we try to define $\Phi^\star : \Mb^0\times\real_+  \arr \Mb^0$ by 
defining $\Phi_i^\star (\XX,t) = \yy_i$, $t \ge S_i$, it may the case that another node $N_j$ with state in $M_j^0$ 
requires $N_i$ to reach a state in $M_i^+\sm \Tr_i$ in order that $N_j$ reach a terminal state (for example, release of a constraint on $N_j$). 
Similar considerations hold for trajectories entering $M_i^0$.
\begin{exam}
\label{term:deadlock}
Take a three node FAN with $M_i = \real$, $\In_i = \{0\}$, $i \in \is{3}$,
$\Tr_1 = \{1\}$, and $\Tr_2 = \Tr_3 = \{2\}$. Define dynamics according to
\begin{align*}
x' & =  1\\
y' & =  \begin{cases}
1,&\;\text{if } y <1,\; \text{or } x \ge  1.5\\
0,&\;\text{if } y = 1, \;\text{and } x < 1.5
\end{cases}\\
z' & = \begin{cases}
\frac{1}{2}, &\; \text{if } x < 1.5\\
1,&\; \text{if } x \ge 1.5
\end{cases}
\end{align*}
If we continue evolution past the terminal states, then
$(0,0,0)$ is $\Phi$-connected to $(1,2,2)$ and
$\mathbf{S}(0,0,0) = (1,2.5,2.75)$. On the other hand, if we stop
evolution of nodes when they reach their terminal state, then 
$(0,0,0)$ is not $\Phi$-connected to $\Tr$: $N_2$ never attains 
its terminal state and there is a deadlock. Moreover, $N_3$ now 
takes time $4$ to reach its terminal state.
\examend
\end{exam}
\subsection{Weak regularity}
\label{wk:rgular}

Our aim is to give conditions on a FAN $\FAN = (\Net,\In,\Tr)$ that allow for
(a) general initializations of nodes in $\In$: state $\xx_i \in \In_i$ starts at any time $T_i \ge 0$, $i\in\is{k}$,  
(b) stopping of nodes when the termination state is reached without
changing the network function. and (c) replacing $\FAN $ by simpler FAN
where the event structure is trivial outside $\Mb^0$.

We need some new notation and definitions. 
Let $\alpha$ be a connection structure. If $i \in\is{k}$, the node $N_i$ is \emph{linked in $\aa$} if
$\aa$ has an edge containing $N_i$ as an end point. Let $v(\aa)$ denote the set of nodes linked in $\aa$.

Define
\begin{align*}
E_i&=\set{\XX \in \Mb}{i \in v(\Ee(\XX))}, \;i \in \is{k}, \\
E^\star&= \bigcup_{i \in \is{k}} E_i.
\end{align*}
Observe that $\XX\in E^\star$ if and only if at least one node is linked
at~$\XX$.  We have
$E^\ECS  = \Mb\smallsetminus E^\star$.

Let
$A^0, A^-, A^+$ be disjoint subsets of  $\is{k}$, with $A^0 \cup A^- \cup A^+ = \is{k}$
(at least two of $A^0, A^-, A^+$ must be nonempty).

If $\sigma \in \{-,+,0\}$ and $\aa \in \A$, define $\aa^\sigma \in \CC{k}$ by
\begin{equation}
\label{alphasig}
\aa^\sigma = \{N_j \arr N_i \in \aa \dd j \in (A^\sigma)^\bullet,\; i \in A^\sigma\},
\end{equation}

Let $\sigma \in \{-,+,0\}$ and suppose that for $i \in A^\sigma$, we are given an 
open neighbourhood $W_i^\sigma$ of $M_i^\sigma$ in $M_i$.  Set $\is{W}^\sigma = \prod_{i \in A^\sigma} W_i^\sigma$
and define
\[
\is{W} = \is{W}^0 \times \is{W}^- \times \is{W}^+ \subset \Mb.
\] 
\begin{Def}
(Notation and assumptions as above.)
The asynchronous network $\Net$ has \emph{product structure} on $\is{W}$ if for each $\sigma \in \{-,+,0\}$,
we can find an asynchronous network $\Net^\sigma = (\Nn^\sigma, \A^\sigma,\F^\sigma,\Ee^\sigma)$,
where $\Nn^\sigma$ has nodes $\{N_i \dd i \in A^\sigma\}$ and network phase space $\is{W}^\sigma$,
such that
\[
\Net | \is{W} = \Net^0 \times \Net^- \times \Net^+.
\]
\end{Def}

If $\Net$ has product structure on $\is{W}$, it follows from the results of~\cite[\S 6]{BF0},
that 
for all $\XX =(\XX^0,\XX^-, \XX^+) \in \is{W}$ we have
\begin{eqnarray}
\label{S1}
\Ee(\XX)&=&\Ee^0(\XX^0) \vee \Ee^-(\XX^-)\vee \Ee^+(\XX^+) \\
\label{S2}
\f^{\Ee(\XX)} (\XX) & = & \f_0^{\Ee^0(\XX^0)}(\XX^0) \times \f_-^{\Ee^+(\XX^-)}(\XX^-) \times \f_+^{\Ee^+(\XX^+)}(\XX^+).
\end{eqnarray}
Moreover, for each $\XX \in \is{W}$, there exists $t(\XX) > 0$ such that the forward trajectory $\Phi^{\is{W}}(\XX,t)$ in $\is{W}$ is well-defined and equal to $(\Phi|\is{W})(\XX,t)$, for all
$t \in [0,t(\XX))$. 
\begin{rems}
\label{lprod:rem}
(1) With the notation of~\Ref{alphasig}, equation~\Ref{S1} implies that $\Ee^\sigma(\XX^\sigma) = \Ee(\XX)^\sigma$, for all $\XX =(\XX^0,\XX^-, \XX^+) \in \is{W}$,
$\sigma \in \{-,+,0\}$. \\
(2) Equation~\Ref{S2} implies that if $i \in A^\sigma$, then 
the component
$f_i^{\Ee(\XX)}|\is{W}$ depends only on $\XX^\sigma \in \is{W}^\sigma$. 
\end{rems}

\begin{Def}
\label{def:wkreg}
(Notation as above.)
The FAN $\FAN =(\Net,\In,\Tr) $ is \emph{weakly regular}
if
\begin{enumerate}
\item  $\FAN$ is geometric. 
\item 
There exist open neighbourhoods $V_i$ of 
$\partial M_i^0$ in $M_i$, $i \in \is{k}$, such that 
\begin{enumerate}
\item[(S1)] $\pi_i(E_i)\subset M_i \smallsetminus V_i$, all $i \in \is{k}$.
\item[(S2)] If $A^0, A^-, A^+$ are proper disjoint subsets of  $\is{k}$, with $A^0 \cup A^- \cup A^+ = \is{k}$, 
and $\is{W}^\sigma = \prod_{i\in A^\sigma} (M_i^\sigma \cup V_i)$, $\sigma \in \{0,-,+\}$, then
$\Net$ has product structure on $\is{W} = \is{W}^0 \times \is{W}^- \times \is{W}^+$.
\end{enumerate}
\end{enumerate}
\end{Def}
\begin{rems}
\label{wr-reg-rems}
(1) If $\aa\in\A$ and $i \in v(\aa)$, then $E_\aa\subset E_i$. Hence (S1) implies that $\pi_i(E_\aa) \subset M_i \sm V_i$ if $i \in v(\aa)$.\\
(2) Condition (S1) implies that if the state of node $N_i$ is close to $\In_i \cup \Tr_i$, 
then the node will be uncoupled.  It also follows from (S1) that 
$\prod_{i \in \is{k}} V_i \subset E^\ECS$ and so $E^\ECS$ is a neighbourhood of
$\In \cup \Tr$. \\
(3) It follows from (S1,S2) and remarks~\ref{lprod:rem}(2) that
$\A^\sigma \subset \A$, $\sigma \in \{0,-,+\}$ (for example, if $\sigma = 0$ and $A^0 \ne \emptyset$, choose $\XX^- \in \prod_{i\in A^-} V_i$, $\XX^+ \in  \prod_{i\in A^+} V_i$).\\
(4) Let $\sigma,\eta \in \{0, -, +\}$, $\sigma \ne \eta$. Since $\Net | \is{W} = \Net^0 \times \Net^- \times \Net^+$, there are no connections between $N_i$ and $N_j$, if $i \in A^\sigma$, $j \in A^\eta$.
\end{rems}

The next result will be crucial
for developing the dynamical and structural properties of weakly regular FANs.

\begin{lemma}[Local product structure]
\label{lem:lps}
Let $\FAN = (\Net,\In,\Tr)$ be a weakly regular FAN.
Let $A \subset \is{k}$ be a nonempty subset of $\is{k}$ and set $B = \is{k}\smallsetminus A$. Set $\Mb_B^\pm =  \prod_{i \in B} (M_i^- \cup M_i^+)$. If $\XX = (\XX_A,\XX_B) \in \Mb^0_A \times \Mb^\pm_B$,
there exists $\delta > 0$ such that for $t \in [0,\delta]$ we may write
\[
\Phi(\XX,t) = (\Phi_A(\XX_A,t),\Phi_B(\XX_B,t)).
\]
\end{lemma}
\begin{proof} 
Choose $\delta > 0$ so that for $t \in [0,\delta]$,
$\Phi_i(\XX,t) \in M_i^0 \cup V_i$,  all $i \in A$, and $\Phi_j(\XX,t) \in M_j^-\cup M_j^+ \cup V_j$, all $j \in B$.
It follows from (S2) that we may write
\[
\Phi(\XX,t) = (\Phi_A(\XX_A,t),\Phi_B(\XX_B,t)) \in \Mb^0_A \times \Mb^\pm_B, \; t \in [0,\delta],
\]
where $\Phi_A(\XX_A,t) = \pi_A \Phi(\XX,t)$, $\Phi_B(\XX_B,t) = \pi_B \Phi(\XX,t)$.
(These relations may fail once $\Phi_i(\XX,t)$ exits $V_i$, $i \in \is{k}$.)
\end{proof}

\subsection{Generalized initialization.}
\label{evo}
Let $\FAN = (\Net,\In,\Tr)$ be a weakly regular FAN with semiflow~$\Phi$. 
Set $\InH = \In \times \real^k_+$, $\TrH =  \Tr \times \real^k_+$. We refer to
$\InH$ and $\TrH$ as \emph{generalized} initialization and  termination sets.  

Our first step will be 
to construct \emph{evolution} and \emph{timing} operators 
\[
\begin{aligned}
\EvOp& = (\EvOp_i,\dotsc,\EvOp_k): \InH \times \real_+\arr \Mb \\
\Sw& = (\sw_1,\dotsc,\sw_k):  \HDD \subset\InH  \arr \real^k_+ 
\end{aligned}
\]
that allow for general initialization times.

Let $\is{T} = (T_1,\dotsc,T_k) \in \real_+^k$. Choose the minimal 
sequence $S(\is{T}) = (\tau_j)_{j=1}^p$ satisfying 
$0 \le \tau_1 < \tau_2 < \dotsc < \tau_p$,  
$\tau_1 = \min \sset{T_1, \dotsc, T_k}$, $\tau_p = \max \sset{T_1,\dotsc,T_k}$, and
$\sset{T_1,\dotsc,T_k} = \sset{\tau_1,\dotsc,\tau_p}$.
For each $\ell \in \is{p}$, 
define $J_\ell =\{ i\in \is{k}\dd T_i \le \tau_\ell\}$. 
\begin{prop}
\label{switch}
(Notation and assumptions as above.)
Let $\XX\in \In$ and $\is{T} \in \real_+^k$.
There is a (unique) continuous, piecewise smooth map $\EvOp_{(\XX,\is{T})}:\real_+ \arr \Mb$ 
satisfying 
\begin{enumerate}
\item For all $i \in \is{k}$, $\EvOp_{i,(\XX,\is{T})}(t) = x_i$, $ t \in [0,T_i]$.
\item If $ t \in [\tau_\ell,\tau_{\ell+1}]$, $\ell < p$, then 
\[
\EvOp_{i,(\XX,\is{T})}(t) = \begin{cases}
&\Phi_i(\EvOp_{(\XX,\is{T})}(\tau_\ell),t-\tau_\ell),\; i \in J_\ell\\
& x_i,\;\text{otherwise}.
\end{cases}
\]
\item If $t \ge \tau_p$, then 
\[
\EvOp_{(\XX,\is{T})}(t) = \Phi(\EvOp_{(\XX,\is{T})}(\tau_p),t-\tau_p).
\]
\item If $\is{T} = \is{0}$, then
\[
\EvOp_{(\XX,\is{T})}(t) = \Phi(\XX,t),\; t\in \real_+.
\]
\end{enumerate}
Finally, if we let $H(t)$ denotes the Heaviside step function (with $H(0) = 1$), 
then the  trajectory $\EvOp_{(\XX,\TT)}(t)$, $t \ge 0$, is the solution  $\XX(t)$ of the non-autonomous system
\begin{equation}
\label{hea}
\xx_i'(t)  =  H(T_i - t)F_i(\XX(t)), \; i \in \is{k},
\end{equation}
where $\is{F} = (F_1,\dotsc,F_k)$ is the network vector field.
\end{prop}
\proof  The trajectory $\EvOp_{(\XX,\TT)}(t)$ is defined inductively using (1,2,3) of the statement. 
Each of the inductive steps gives a well-defined piece of trajectory by lemma~\ref{lem:lps}.
Once all the variables are switched on, we use the properness of $\Net$ to
define $\EvOp_{(\XX,\TT)}(t) = \Phi(\YY,t-\tau_p)$, $ t \ge \tau_p$, where $\YY = \EvOp_{(\XX,\is{T})}(\tau_p)$.
The remaining statements of the proposition are immediate. \qed
\begin{rem}
Proposition~\ref{switch} shows that we can start to evolve nodes from the initialization set at different times. Observe 
that the proof depends crucially on the local product structure given by lemma~\ref{lem:lps}. In particular,
without the local product structure there is no guarantee that solutions to~\Ref{hea} exist in the sense of
definition~4.16~\cite{BF0}.  
\end{rem}

\begin{cor}
\label{lem:tf} 
Let $\FAN =(\Net,\In,\Tr)$ be weakly regular. 
The evolution operator $\EvOp: \InH \times \real_+ \arr \Mb$ given by proposition~\ref{switch}
is well defined and continuous in forward time. 
\end{cor}

For $i \in\is{k}$, 
let $\mathcal{D}_i$ be the subset 
of $\InH$ consisting of $(\XX,\TT)$ for which there exists
a (minimal) $\sw_i = \sw_i(\XX,\TT) \geq T_i$ such that 
$\EvOp_i((\XX,\TT),\sw_i)\in\Tr_i$. 
Set $\HDD = \bigcap_{i\in\is{k}} \mathcal{D}_i$ and define
the \emph{timing function} $\Sw: \HDD \arr \real^k_+$ by
$\Sw(\XX,\TT) = (\sw_1(\XX,\TT),\dotsc, \sw_k(\XX,\TT))$.
Since $\FAN$ is weakly regular, $ \HDD \supset \In$ 
(we identify~$\In$ with the subset $\{(\XX,\is{0})\dd \XX \in \In\}$ of $\InH$).

\begin{Def}
(Notations and assumptions as above.)
If $\HDD = \InH$, we define the \emph{generalized transition function}
$G: \InH \arr \TrH$ by
\[
G(\XX,\TT) = (
(\EvOp_1((\XX,\TT), \sw_1),\dotsc,\EvOp_k((\XX,\TT), \sw_k)),
\Sw(\XX,\TT)
),
\]
where $\Sw(\XX,\TT) = (\sw_1,\dotsc,\sw_k)$ is given by the timing function. 
\end{Def}

\begin{rem}
For $(\XX, \TT)\in\InH$ with $\TT = (T, \dotsc, T)$, $T\in\real_+$
we have $G(\XX, \TT) = (G_0(\XX), \Ss(\XX)+\TT)$.
\end{rem}

\begin{Def}
(Notations and assumptions as above.)
A weakly regular FAN is \emph{regular} if $\HDD = \InH$. 
\end{Def}

\begin{rem}
\label{altnot}
If $\FAN=(\Net,\In,\Tr)$ is weakly regular, we write $\FANw=(\Net,\InH,\TrH)$ 
to emphasise that we require initialization from $\InH =\In \times \real^k_+$
rather than at time zero from~$\In$. In particular, the FAN $\FANw$ achieves 
its network function (that is, $\FAN$ is regular) if
every point of $\InH$ is $\EvOp$-connected to $\TrH$.
\end{rem}
 
\begin{exam}
The example of two trains on a single track railway line with a 
passing loop and stations described in \cite[\S 5.1]{BF0}. 
admits a generalized transition function. 
\end{exam}

\subsection{Hidden deadlocks}

\begin{Def}
Let $\FAN=(\Net,\In,\Tr)$ be a weakly regular FAN with semiflow $\Phi$. A compact
$\Phi$-invariant set $A\subset\Mb^0$ is a \emph{hidden deadlock}
if
\begin{enumerate}
\item  $A \cap \Tr = \emptyset$.
\item $A$ is a deadlock for the FAN $\FANw$. 
\end{enumerate}
\end{Def}

\begin{rems}
(1) Condition (2) of the definition is equivalent to there existing $(\XX,\TT) \in \InH$ and $t > 0$ such that
$\EvOp((\XX,\TT),t) \in A$. Note that by the $\Phi$-invariance of $A$,
$\pi_i(A) \cap \In_i = \emptyset$, all $i \in \is{k}$. Hence, once a $\EvOp$-trajectory
has entered $A$, the subsequent evolution of the nodes 
is given by~$\Phi$ and so the $\EvOp$-trajectory cannot
leave~$A$.  \\
(2) Since $\FAN=(\Net,\In,\Tr)$ is assumed weakly regular, a hidden deadlock 
can never be a deadlock of $\FAN$.
\end{rems}

\begin{figure}[h]
\centering
\includegraphics[width=0.8\textwidth]{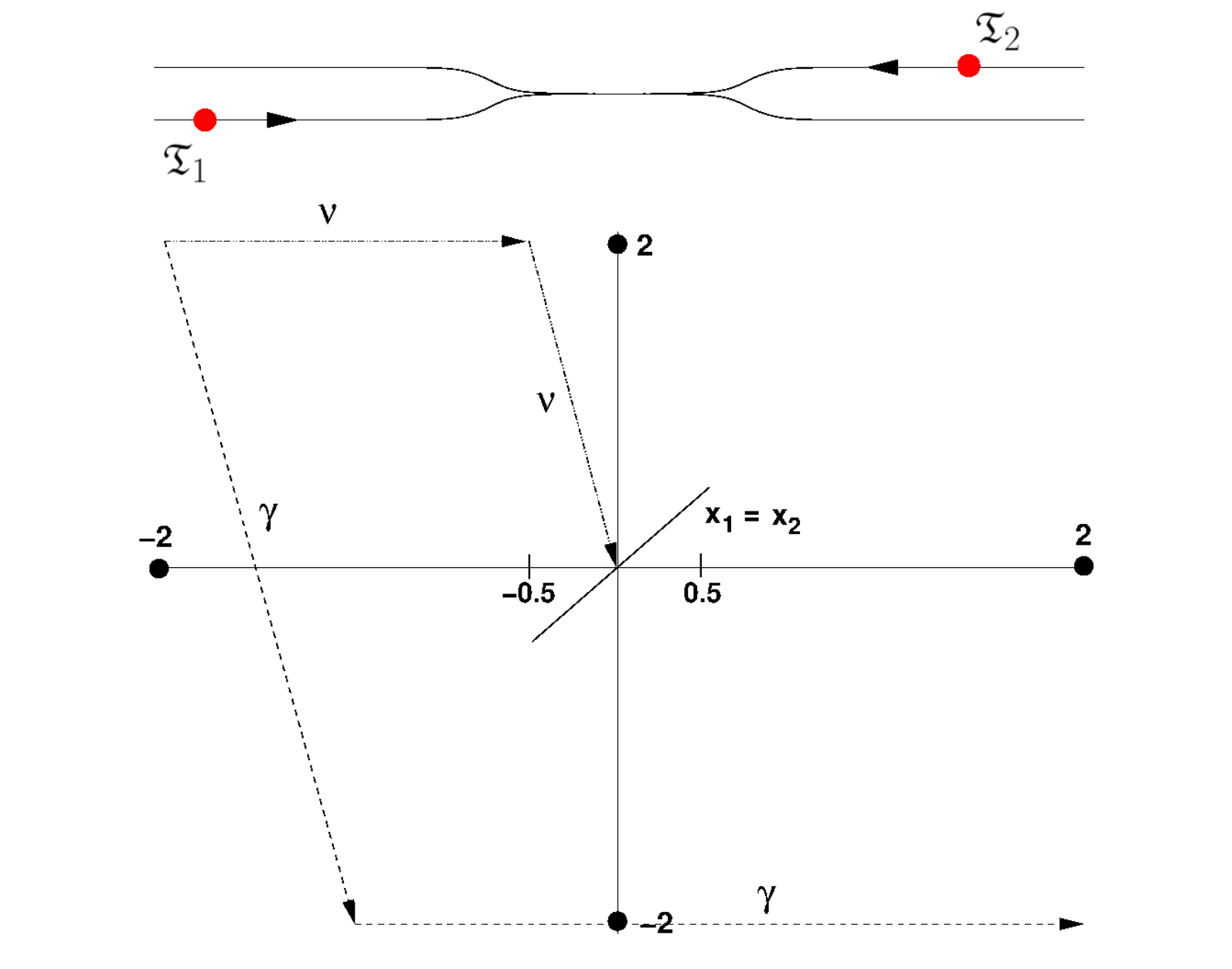}
\caption{Two trains on a partially single track line.}
\label{timings}
\end{figure}
\begin{exam}
\label{ex:SingleTrack}
Referring to figure~\ref{timings}, consider two trains travelling in the
opposite direction on a line which contains a single track segment
$[-0.5,0.5]$ and stations at $\pm 2$. We associate a collision event with the single track segment which
results in both trains stopping.
Train $\mathfrak{T}_1$
starts at $\sset{-2}=\In_1$ and proceeds with velocity~$0.5$;
train $\mathfrak{T}_2$ starts at $\sset{2}=\In_2$ and proceeds
with velocity~$-2.0$. The termination sets are
$\Tr_1 = \{2\}$, $\Tr_2 = \{-2\}$. The trajectory $\gamma$ of figure~\ref{timings}, represents the case where
both trains start at time $t = 0$.
At $t = 1.25$, $\mathfrak{T}_1$ is at point $-1.375$,
and $\mathfrak{T}_2$ is at $-0.5 > -1.375$. Hence there is no collision 
(deadlock) and 
network function is achieved.
On the other hand if $\mathfrak{T}_2$ starts at time $T_2 = 3$ and $\mathfrak{T}_1$  at time $T_1 = 0$,
the trains will collide at the origin at time $t = 4$ (the trajectory~$\nu$ 
of figure~\ref{timings}) and the network is not regular.
\examend
\end{exam}
 
\subsection{FANs of simple type}
The definitions of regularity and weak regularity for a FAN involve network data not directly related to the
network function. We show that given a
weakly regular FAN~$\FAN$, we can construct a simpler variant~$\FAN^c$,  the \emph{core} of $\FAN$, which has the same function as $\FAN$ but
carries only the essential structure of $\FAN$ needed for construction of the generalized transition function. Roughly speaking, we are making the transition 
from viewing the FAN as a (possible) subnetwork of a larger network (relative viewpoint) to an absolute viewpoint
(the FAN is not contained in a larger network).

\begin{Def}
A geometric FAN $\FAN =(\Net,\In,\Tr) $ is of \emph{simple type} if 
for all $i \in \is{k}$, we can choose open neighbourhoods $V_i$ of $\partial M_i^0$ such that 
\begin{enumerate}
\item[(S1)$^c$] $\pi_i(E_i) \subset M_i^0\smallsetminus V_i$.
\item[(S2)$^c$] 
\begin{enumerate}
\item[(a)] The event map $\Ee(\xx_1,\dotsc,\xx_k)$ is locally 
constant as a function of $\xx_i \in M_i^\sigma\cup V_i$, where 
$\sigma \in\{+,-\}$.
\item[(b)] For all $i \in \is{k}$, $f_i^{\Ee(\XX)}(\XX) = Z_i(\xx_i)$ if $\xx_i \in M_i^- \cup M_i^+ \cup V_i$.
\end{enumerate}
\end{enumerate} 
\end{Def}

\begin{rem}
Condition~(S2)$^c$ relates to the concepts of 
structural decomposability, see~\cite[\S 6]{BF0}. 
If (S2a)$^c$ holds, there are no dependencies 
of the event map on the state of a 
node~$N_i$ once its state has exited $M_i^0$.
\end{rem}

\begin{lemma}
If the FAN $\FAN =(\Net,\In,\Tr) $ is of simple type, then
$\FAN$ is weakly regular. If, in addition, $\FAN$ is regular, we say
$\FAN$ is \emph{regular of simple type}.
\end{lemma}
\begin{proof} 
Obviously (S1)$^c$ implies (S1) and (S1,S2)$^c$ imply (S2). 
\end{proof}

We need some new notation before stating our next result.  Suppose that $\FAN =(\Net,\In,\Tr) $ is a weakly regular FAN.
For $\XX \in \Mb$, define mutually disjoint subsets $A^0(\XX), A^\pm(\XX)$ of $\is{k}$ by
\begin{align*}
A^0(\XX) & = \set{i \in \is{k} }{ \xx_i \in M_i^0 \smallsetminus \partial M_i^0},\\
A^-(\XX) & = \set{i \in \is{k} }{ \xx_i \in M_i^-},\\
A^+(\XX) & = \set{i \in \is{k} }{ \xx_i \in M_i^+}.
\end{align*}
We have $\is{k} = A^0(\XX) \cup A^-(\XX)\cup A^+(\XX)$ for all $\XX\in \Mb$.

For each $i \in \is{k}$, fix $\is{c}_i^- \in \In_i$,
and $\is{c}_i^+ \in \Tr_i$. 
Given $\XX \in \Mb$, define $\widetilde{\XX}\in \Mb$ by
\begin{equation*}
\widetilde{\xx}_i = \begin{cases}
\is{c}^-_i,&\;\text{if } i \in A^-(\XX),\\
\xx_i, &\;\text{if } i \in A^0(\XX),\\
\is{c}^+_i,&\;\text{if } i \in A^+(\XX).
\end{cases}
\end{equation*}

Define the event map $\Ee^c: \Mb \arr \A$ by
\[
\Ee^c(\XX) = \Ee(\widetilde{\XX}), \;\XX \in \Mb.
\]
Define $\A^c = \tset{\Ee^c(\XX)}{\XX \in \Mb} \subset \A$ and $\F^c = \tset{\f^\alpha}{\alpha \in \A^c} \subset \F$.
Note that, by weak regularity, $\Ee^c$, and hence $\A^c$ and $\F^c$, do not depend on the specific choice of $\is{c}_i^\pm$ 
used for the definition of $\widetilde{\XX}$. 
Let $\Net^c$ denote the asynchronous network $(\mathcal{N},\A^\s,\F^\s,\Ee^\s)$.

\begin{thm}
\label{thm:simp}
The asynchronous network $\Net^c$ is proper and $\FAN^c =(\Net^c,\In,\Tr) $ 
is a FAN of simple type. If we denote the evolution operator of
$\FAN^c$ by $\EvOpc$ then
for all  $((\XX,\TT),t) \in \InH \times \real_+$ we have
\begin{equation}
\label{eq:core}
\EvOpc_i((\XX,\TT),t) = \EvOp_i((\XX,\TT),t),\;\text{if } \EvOp_i((\XX,\TT),s) \in M_i^0, s\le t.
\end{equation}
If $\FAN$ is regular then so is~$\FAN^c$ and both FANs have the same 
generalized transition function. 
\end{thm}
\begin{proof} The properness of $\Net^c$, weak regularity of $\FAN^c$ and \Ref{eq:core} 
are immediate from the construction of $\FAN^c$.
The remaining statements follow from~\Ref{eq:core}. \end{proof}

\begin{Def}
We refer to the FAN $\FAN^\s =(\Net^\s,\In,\Tr) $ given by theorem~\ref{thm:simp} as the \emph{core}
of  $\FAN= (\Net,\In,\Tr) $ and write 
$\FAN^\s = \text{core}(\FAN)$.
\end{Def}
\subsection{Stopping nodes}
\label{sec:Stopping2}
Let $\FAN = (\Net,\In,\Tr)$ be a weakly regular FAN with associated
timing function $\is{S} = (S_1,\dotsc,S_k): \In \arr \real_+$. In this section
we address the issue raised earlier in the section of stopping nodes when they reach their termination sets. 

We first construct a new FAN $\FAN^\star = (\Net^\star,\In,\Tr)$ that
provides a simple minimal model for the network function of~$\FAN$ and
for which nodes are stopped when they reach their terminal state.

We start by formalizing the process of stopping of nodes.  
Let $A$ be a nonempty subset of $\is{k}$ and $\xi_A$ be the connection structure
$\{N_0 \arr N_i \dd i \in A\}$. Define
$\A^s = \A \cup \set{\xi_A}{A \subset \is{k}}$.  We define an  event
map $\Ee^s: \Mb \arr \A^s$. Let $\XX \in \Mb$. If $\xx_i \in M_i^+$,
$i \in A$, and $\xx_i \notin M_i^+$, $i \notin A$,  define
\[
\Ee^s(\XX) = \Ee(\XX) \vee \xi_A.
\]
For each $\alpha^s = \alpha \vee \xi_A \in \A^s$, define $\f^{\alpha^s}=(f_1,\dotsc,f_k)$ by
\[
f_i(\XX) = \begin{cases}
0, & \text{if } i \in A\\
f_i^\alpha(\XX), & \text{if } i \notin A. 
\end{cases}
\]
Define $\F^s = \{ \f^{\alpha^s} \dd \alpha^s \in \A^s\}$.

Define the asynchronous network $\Net^\star = (\Nn^\star,\A^\star, \F^\star, \Ee^\star)$
by requiring that $\Net^\star$ and $\Net$ have the same node sets ($\Nn^\star = \Nn$) but
take the network phase space of $\Net^\star$ to be $\Mb^0$.
Define $\Ee^\star$ by restriction of $\Ee^s$
to $\Mb^0$ and set $\A^\star = \A^s$ (equal to $\{\Ee^\star(\XX) \dd\XX \in \Mb^0\}$),  and
$\F^\star = \tset{\f^\alpha|\Mb^0}{ \alpha \in \A^\star}$.
\begin{prop}
\label{stop1}
(Notation and assumptions as above.)
The asynchronous network $\Net^\star$ is proper, with well-defined semiflow
$\Phi^\star : \Mb^0 \times \real_+ \arr \Mb^0$.
For all $\XX\in\Mb^0$, $i \in \is{k}$, we have
\begin{equation}
\label{equal}
\Phi_i^\star(\XX,t) = \begin{cases}
\Phi_i(\XX,t),&\text{if } t \le S_i,\\
\Phi_i(\XX,S_i),&\text{if }  t \ge S_i.
\end{cases}
\end{equation}

In particular,
\begin{enumerate}
\item $\FAN^\star$ is a FAN, that is, $\In,\Tr$ satisfy conditions (P1--4)
for~$\Net^\star$.
\item  $\FAN$ and~$\FAN^\star$ have the same
transition and timing functions.
\end{enumerate}
\end{prop}
\begin{proof} A routine computation based, as usual, on lemma~\ref{lem:lps}.  \end{proof}

\begin{rem}
Proposition~\ref{stop1} may fail if $\FAN$ does not satisfy the structural conditions (S1,S2) for weak regularity.
Indeed, $\Net^\star$ may not be proper; even if it is,
\Ref{equal} may fail.
\end{rem}

Proposition~\ref{stop1} shows that for a weakly regular FAN  $\FAN = (\Net,\In,\Tr)$,
stopping nodes at termination does not change network
function. This property also holds for generalized initialization.
\begin{prop}
\label{stop2}
Let $\FAN =(\Net,\In,\Tr)$ be weakly regular.
If we let $\EvOp^\star$ denote the evolution operator for $\FAN^\star$,
then
\begin{enumerate}
\item $\EvOp^\star: \InH \times \real_+ \arr \Mb^0$ is well defined  and continuous in forward time.
\item  For all $i \in \is{k}$, $((\XX,\TT),t) \in \InH \times \real_+ $, we have
\[
\EvOp^\star_i((\XX,\TT),t) = \EvOp_i((\XX,\TT),t), \; \text{if } \EvOp_i((\XX,\TT),t) \in M_i^0.
\]
\item If we denote the timing operator for $\EvOp^\star$ by $\Sw^\star$,
then $\Sw = \Sw^\star$ and the operators have common domain~$\HDD$.
\end{enumerate}
If $\FAN$ is regular, then $\FAN$ and $\FAN^\star$ have identical generalized transition
and timing functions.
\end{prop}
\begin{proof}
Another application of lemma~\ref{lem:lps}.
\end{proof}
\begin{rems}
(1) For weakly regular FANs with generalized initialization, stopping
of nodes upon completion has no effect on the dynamics of the other nodes
whose states are in~$M_i^0$. As a result, the evolution
operator~$\EvOp^\star$ suffices for the analysis of network function even
if nodes may not terminate for some generalized initial conditions.\\
(2) If $\FAN$ is regular, every point $\InH$ is $\EvOp$-connected to $\TrH$
and the $\EvOp^\star$-trajectory of every point in $\InH$ meets $\TrH$
(note our use
of the notation $\FANw = (\Net,\InH,\TrH)$ for a weakly regular FAN, remark~\ref{altnot}).
\end{rems}

We conclude this section with some additional remarks and comments about the conditions of definitions~\ref{def:geom} and \ref{def:wkreg}.

\begin{rems}
\label{if2}
(1) In the sequel it will sometimes be convenient to take $\In_i = \Tr_i$ for some indices $i$. We then have $M_i^0 = \In_i = \Tr_i$ and require that the vector field
$Z_i$ point from $M_i^-$ to $M_i^+$.  \\
(2) Typically, different choices of initialization and 
termination sets satisfying (G,T,F) will be isotopic (by the flow 
of the uncoupled node). Thus, if $\In_i, \In_i'$ are 
initialization sets for $M_i$, there will exist a smooth map 
$\xi: \In_i \arr \real$ such that 
$\In_i' = \tset{ \psi_i^{\xi(\xx)}(\xx)}{\xx \in \In_i}$. 
Similarly for the termination hypersurfaces $\Tr_i$.  We 
use this property later.\\
(3) We do not require that the flow $\psi_i^t$ maps $\In_i$ 
to $\Tr_i$ -- for example, the vector fields $Z_i$ may have 
equilibria in $M_i^0$. However, interactions with other nodes will 
then be needed for the state of $N_i$ to reach the termination 
set~$\Tr_i$. If~$N_i$ is never coupled to other nodes, then 
(F) implies that $\psi_i^t$ maps $\In_i$ to $\Tr_i$. 
\end{rems}

\section{Amalgamation and concatenation of FANs of simple type}
\label{join:sec}

In this section we define the operations of amalgamation and concatenation of FANs that share a common node set.
Throughout, we assume all FANs are of simple type.
\subsection{Preliminaries}

Let $\Net = (\Nn,\A,\F,\Ee)$ be an asynchronous network. 
Recall from section~\ref{wk:rgular} that if $\aa \in \A$, then $v(\aa) \subset \is{k}$ is the set of
nodes linked in $\aa$. 
Define
\[
V(\A) = \bigcup_{\alpha \in \A} v(\alpha)\subset \is{k}.  
\]
\begin{Def}
The asynchronous network $\Net$ 
is \emph{weakly input consistent} if $\ECS\in\A$ and $i \notin v(\alpha)$ implies that
$f_i^\aa = f_i^\ECS$.
\end{Def}
We always assume asynchronous networks are weakly input consistent.

We easily extend the definition of a product of asynchronous networks given in~\cite[\S 6]{BF0}, to FANs. Thus, if
$\FAN^a = (\Net^a,\In^a,\Tr^a)$, $a \in \is{q}$, are FANs (with disjoint node sets), we define
\[
\prod_{a \in \is{q}} \FAN^a = (\prod_{a \in \is{q}} \Net^a,\prod_{a \in \is{q}} \In^a, \prod_{a \in \is{q}} \Tr^a).
\]
Next we define some basic building blocks.
\begin{Def}
Let $\FAN = (\Net,\In,\Tr)$ be a FAN of simple type with $k$ nodes.
\begin{enumerate}
\item $\FAN$ is \emph{trivial} if $\FAN = \prod_{\ell\in\is{k}}\is{S}^\ell$, where
$\is{S}^\ell$ has one node for all $\ell \in \is{k}$.
\item $\NN$ is \emph{indecomposable} if $k> 1$ and $\NN$ cannot be written as a product of FANs.
\item $\NN$ is \emph{stably indecomposable} if $k \ge 3$ and there
is an indecomposable FAN~$\is{P}$ and a trivial
FAN $\is{S} = \prod_{\ell\in\is{s}}\is{S}^\ell$ such that
$\FAN = \is{P}\times \is{S}$.
\item $\NN$ is \emph{elementary} if $\NN = \prod_{i \in \is{p}}\is{P}^i \times \prod_{\ell\in\is{q}}\is{S}^\ell$,
where the $\is{P}^i$ are indecomposable,  the $\is{S}^\ell$ are trivial, and $p,q \ge 0$.
\end{enumerate}
If $\FAN $ is a elementary FAN, let
$A(\FAN) \subsetneq\is{k}$ denote the set of nodes associated with the indecomposable factor(s).
\end{Def}
\begin{rems}
(1) If $\FAN$ is a stably indecomposable FAN with indecomposable 
factor $\is{P}$, then $V(\A)$ may be a proper subset of $A(\FAN)$.
That is, there may be nodes $N_i$ in $\Nn_P$ that are
never linked: $i \notin V(\A)$. Since $\is{P}$ is indecomposable, this means that
certain values of $\xx_i\in M_i^0$ may result in connections between other
nodes being switched on or off. \\
(2) If $\FAN$ is a stably indecomposable FAN with indecomposable
factor $\is{P}$, then for all $\alpha \in \A$,  
there exists $\hat{E}^\alpha \subset \Mb_{A(\FAN)}$ 
such that $E^\aa = \hat{E}^\alpha \times \Mb_{\is{k}\sm A(\FAN)}$.
\end{rems}

For the remainder of this section, assume that all FANs are of simple type and share 
\begin{enumerate}
\item a common node set $\Nn = \{\NS,N_1,\dotsc,N_k\}$; 
\item a common network phase space $\Mb = \prod_{i\in\is{k}} M_i$;
\item a common $\ECS$-admissible network vector field $\f^\ECS$. 
\end{enumerate}
\subsection{Amalgamation}

\begin{Def}
\noindent Elementary FANs $\FAN^a, \FAN^b$ are \emph{independent} if 
\begin{enumerate}
\item $A(\FAN^a) \cap A(\FAN^b) = \emptyset$. 
\item $M_i^{a,\sigma} = M_i^{b,\sigma}$, $\sigma \in \{+,-,0\}$. In particular, $\FAN^a, \FAN^b$ have the same 
initialization and termination sets.
\end{enumerate}
\end{Def}

\begin{prop}
\label{prop:amal}
Let $\Lambda =\{\FAN^a=(\Net^a,\In,\Tr) \dd a \in \is{q}\}$ be a family of elementary FANs such that for all $a,b\in \is{q}$, $a \ne b$,
$\FAN^a$ and $\FAN^b$ are independent. Then there is a unique proper elementary FAN $\FAN=(\Net,\In,\Tr)$ characterised by
\begin{enumerate}
\item $\A = \bigvee_{a\in \is{q}} \A^a$.
\item $\F = \{\f^\alpha \dd \alpha = \alpha_1 \vee \dotsc \vee \alpha_q \in \A\}$ where, for $i\in\is{k}$,
\[
f_i^\alpha = \begin{cases}
&f_i^{\aa_a}, \;\text{if } i \in v(\aa_a),\\
& f_i^\ECS,\;\text{if } i \not\in \cup_{a\in \is{q}}(v(\aa_a)).
\end{cases}
\]
\item $\Ee(\XX) = \bigvee_{a\in \is{q}} \Ee^a(\XX)$, $\XX\in \Mb$.
\end{enumerate}
If the family $\Lambda$ consists of regular FANs, then $\FAN$ is regular.
\end{prop}
\proof The result follows easily using weak input consistency 
and by noting that $\FAN $ can be written as a product of indecomposable factors, together with the
trivial factors corresponding to nodes in $\is{k}\sm \cup_{a\in\is{q}}A(\FAN^a)$. The FAN $\FAN$ is
obviously of simple type. \qed
\begin{rem}
We denote the FAN constructed in the proposition by $\bigsqcup_{a \in \is{q}} \FAN^a $ and refer to
it as the \emph{amalgamation} of the family $\Lambda$.
\end{rem}
Assume from now on that the FANs of proposition~\ref{prop:amal} are all stably decomposable -- this is no loss of generality
since, by the proposition, an elementary FAN $\FAN$ can be written as an amalgamation of the stably indecomposable FANs determined by
the indecomposable factors of $\FAN$.
For $a \in \is{q}$, let $N^a$ have decomposition $\is{P}^a \times \is{S}^a$, where $\is{P}^a$ is indecomposable and
$\is{S}^a$ is trivial.  Set $A_a = A(\FAN^a)$, $\overline{A}_a = \is{k}\sm A_a$. 
Since $\FAN^a$ is weakly regular, the transition function $G^a_0: \In \arr \Tr$ for $\FAN^a$ 
may be written as the product
\[
G^a_P \times \prod_{\ell \in \overline{A}_a}G_{S,\ell}: 
\In_P^a \times \prod_{\ell\in  \overline{A}_a} \In_\ell \arr \Tr_P^a \times \prod_{\ell\in  \overline{A}_a} \Tr_\ell,
\]
where $\In_P^a = \prod_{j \in A(\FAN^a)}\In_j$, $\Tr_P^a = \prod_{j \in A(\FAN^a)}\Tr_j$ and $G_{S,\ell}$ is independent of
$a$ since $\f^{a,\ECS} = \f^\ECS$, all $a \in \is{q}$.

Summarising, we have a result about the transition functions of amalgamations.
\begin{cor}
Let $\{\FAN^a=(\Net^a,\In,\Tr) \dd a \in \is{q}\}$ be a set of stably indecomposable pairwise indepedent FANs.
The transition function $G_0$ for  $\bigsqcup_{a \in \is{q}} \FAN^a $ may be written uniquely as 
\[
G_0 = \prod_{a\in\is{q}} G^a_P \times \prod_{\ell \in\overline{A}} G_{S,\ell}: 
\prod_{a\in\is{q}}  \In_P^a \times \prod_{\ell\in \overline{A}} \In_i \arr \prod_{a\in\is{q}}  \Tr_P^a \times \prod_{\ell\in \overline{A}} \Tr_i,
\]
where $\overline{A} = \is{k} \sm \cup_{a\in\in{q}} A(\FAN^a)$ parametrizes the trivial factors of $\bigsqcup_{a \in \is{q}} \FAN^a $. 

Similar results hold for the evolution operator and for generalized transition functions if each
FAN $\FAN^a$ is also regular.
\end{cor}

\subsection{Concatenation}
Amalgamation can be viewed as a spatial merging of FANs. We now consider a related variation that is is suggestive of a temporal
merging.
\begin{Def}
Suppose $\FAN^1, \FAN^2$ are elementary FANs (necessarily of simple type).
The FAN $\FAN^1$ \emph{precedes} $\FAN^2$, written $\FAN^1 \prec \FAN^2$, if
\begin{enumerate}
\item $A(\FAN^1) \cap A(\FAN^2) \ne \emptyset$.
\item If $i \in A(\FAN^1) \cup A(\FAN^2) $, then $\Tr^1_i = \In^2_i= M_i^{1,0}\cap M_i^{2,0}$, and $M_i^{1,-} =M_i^{2,0}\cup M^{2,-}$.
\item If $i \notin  A(\FAN^1) \cup A(\FAN^2)$,  then
$M_i^{1,\sigma} = M_i^{2,\sigma}$, $\sigma \in \{+,-,0\}$.  
\end{enumerate}
\end{Def}

Suppose $\FAN^1 \prec \FAN^2$. 
For $a \in \mathbf{2}$, let $E^a_i = \{\XX \in \Mb\dd i\in v(\Ee^a(\XX))\}$,
$E^{a,\star} = \cup_{i\in\is{k}}E^a_i$, 
and $E^a_\alpha$ be the event set corresponding to $\alpha \in \A^a$.

Define 
\[
\A = \{ \aa_1 \vee \aa_2 \dd \exists \XX \in \Mb,\; \alpha_1 = \Ee^1(\XX),\; \aa_2 = \Ee^2(\XX)\}.
\]
For the moment, it is convenient to regard $\aa_1 \vee \aa_2, \beta_1\vee\beta_2$ as distinct elements of $\A$ if
$\{\aa_1,\aa_2\} \ne \{\beta_1,\beta_2\}$, even if $\aa_1 \vee \aa_2=\beta_1\vee\beta_2$. To emphasise this, we write
$\aa = \aa_1 \ovee \aa_2$ to indicate the  particular decomposition of $\alpha$ as $\aa_1 \vee \aa_2$.

If $\aa_1\ovee\aa_2 \in \A$, let $E_{\aa_1\ovee\aa_2} = E_{\aa_1}^1 \cap E_{\aa_2}^2$.
For $\aa\in \A$, define 
\[
E_\aa = \cup_{\aa_1\vee\aa_2 = \aa} E_{\aa_1\ovee\aa_2}. 
\]
\begin{lemma}
\label{admin0}
(Notation and assumptions as above.) If $\aa_1 \ovee \aa_2\in \A$, then $v(\aa_1) \cap v(\aa_2) = \emptyset$.
In particular, 
\[
v(\Ee^1(\XX)) \cap v(\Ee^2(\XX)) = \emptyset,\;\text{for all } \XX \in \Mb.
\]
\end{lemma}
\proof If $i \in v(\aa_1) \cap v(\aa_2)$, then $\overline{\pi_i(E_{\aa_1}^1)}\cap \overline{\pi_i(E_{\aa_2}^2)} = \emptyset$
by (S1)$^c$. Hence $E_{\aa_1} \cap E_{\aa_2} = \emptyset$, contradicting the hypothesis that $\aa_1 \ovee \aa_2\in \A$.  
\qed
\begin{lemma}
\label{admin1}
If $\aa_1 \vee \aa_2 = \beta_1 \vee \beta_2$, $\aa_1 \ovee \aa_2,\beta_1 \ovee \beta_2 \in\A$, and either $\aa_1 \ne \beta_1$ or $\aa_2 \ne \beta_2$, then 
$\overline{E_{\aa_1\ovee\aa_2}}  \cap \overline{E_{\beta_1\ovee\beta_2}} = \emptyset$. 
\end{lemma}
\proof By lemma~\ref{admin0}, we may assume $v(\aa_1)\cap v(\aa_2) = v(\beta_1)\cap v(\beta_2) = \emptyset$. 
Without loss of generality, suppose $\aa_1 \ne \beta_1$ and $v(\aa_1)\sm v(\beta_1) \ne \emptyset$. Since
$v(\aa_1)\cup v(\aa_2) = v(\beta_1)\cup v(\beta_2)$, $\exists j \in v(\aa_1)\cap v(\beta_2)$.  We have 
\[
\overline{\pi_j(E^1_{\aa_1})} \subset M^{1,0}_j \sm V^1_j,\; \overline{\pi_j(E^2_{\beta_2})} \subset M^{2,0}_j \sm V^2_j.
\]
Since $E_{\aa_1\ovee\aa_2}  \subset E^1_{\aa_1}$ and $E_{\beta_1\ovee\beta_2}  \subset E^2_{\beta_2}$,
we have $\overline{E_{\aa_1\ovee\aa_2}}  \cap \overline{E_{\beta_1\ovee\beta_2}} = \emptyset$. \qed
\begin{lemma}
\label{admin2}
Let $\alpha \in \A$. There exists a smooth weakly input consistent admissible vector field $\f^\alpha$ such that if $\aa = \aa_1 \ovee \aa_2$,
and $\XX \in  \overline{E_{\aa_1\ovee\aa_2}}$, we have
\begin{equation}
\label{coneq}
f_i^{\alpha} (\XX)  = \begin{cases}
&f_i^{j,\aa_j}(\XX), \; \text{if } i \in v(\aa_j),\\
&f^\ECS_i(\XX), \;i \notin v(\aa_1)\cup v(\aa_2).
\end{cases}
\end{equation}
In particular, if $\aa \in \A^1\cap \A^2$, we may choose $\f^\alpha$ so that
$\f^{\alpha} | E^j_\alpha = \f^{j,\alpha}|E^j_\alpha$, $j \in\is{2}$.
\end{lemma}
\proof  It follows by lemma~\ref{admin1}, that we can use~\Ref{coneq} to define $\f^\alpha$ on $\overline{E_\aa}$.
Now use Whitney's extension theorem~\cite{Whitney34}, or a simple partition of unity argument, to extend
$\f^\aa$ to $\Mb$ as a weakly input consistent admissible vector field.  \qed 
\begin{rem}
If $\aa \in \A$ has a unique representation as $\aa_1\vee\aa_2$, then we can use 
\Ref{coneq} to define $\f^\aa$, all $\XX \in \Mb$. 
\end{rem}

We define the $\A$-structure $\F$ to be $\{\f^\aa \dd \aa \in \A\}$ and the event map by
$\Ee(\XX) = \Ee^1(\XX) \vee \Ee^2(\XX)$, $\XX \in \Mb$.

\begin{prop}\label{concat}
(Notation and assumptions as above.) If
$\FAN^1 \prec \FAN^2$, then
there exists a FAN $\FAN = (\Net,\In,\Tr)$ of simple type
where $\A$, $\F$, $\Ee$ are as defined above and $\In = \In^1$, $\Tr = \Tr^2$.\\

\noindent Moreover, 
$\is{F} | \prod_{i\in\is{k}}(M_i^{1,-}\cup M_i^{1,0}) = \is{F}^1$, and
$\is{F} | \prod_{i\in\is{k}}(M_i^{2,0}\cup M_i^{2,+}) = \is{F}^2$, where
$\is{F}$, $\is{F}^1$, and $\is{F}^2$  denote the network vector fields for $\Net$, $\Net^1$ and $\Net^2$
respectively.
\end{prop}
\begin{proof} The result is immediate from our constructions. \end{proof} 

\begin{rems}
(1) We call the FAN given by proposition~\ref{concat} the \emph{concatenation} of $\FAN^1$ and $\FAN^2$
and denote by $ \FAN^2 \conc \FAN^1 $.
Unlike for the amalgamation operation, order matters.  \\
(2) Observe that if $A(\FAN^1) \cap A(\FAN^2) = \emptyset$, then the construction we gave for 
$\FAN^2 \conc \FAN^1$ still works and gives the amalgamation $\FAN^1 \sqcup \FAN^2$.  In this sense, the concatenation
is a generalization of amalgamation. We prefer to keep the two operations separate.
\end{rems}

Suppose that $\FAN^1, \FAN^2$ are regular FANs of simple type and that $\FAN^1 \prec \FAN^2$.
For $j \in \is{2}$, denote the evolution operator, generalized transition function and timing function of $\FAN^j$ by
$\EvOp^j:\Mb^j\times \InH^j\arr \Mb$, $G^j:\InH^j\arr \TrH^j$ and $\Sw^j:\InH^j \arr \real_+^k$ respectively. 
\begin{cor}\label{cor1:conc}
(Notation as above.) 
Suppose that $\FAN^1, \FAN^2$ are regular FANs of simple type and $\FAN^1 \prec \FAN^2$. We have
\begin{enumerate}
\item $ \FAN^2 \conc \FAN^1 $ is a regular FAN of simple type.
\item The evolution operator $\EvOp$, generalized transition function $G$ and timing function $\Sw$ of $ \FAN^2 \conc \FAN^1 $
satisfy
\begin{enumerate}
\item For all $i\in \is{k}$, $(\XX,\is{T}) \in \InH$, $t \in \real_+$,
\[
\EvOp_i((\XX,\TT), t)  = 
\begin{cases}
\EvOp^1_i((\XX,\TT), t), \; t \le \sw^1_i\\
\EvOp^2_i((G^1(\XX,\TT),\Sw^1(\XX,\TT)), t), \; t \ge \sw^1_i.
\end{cases}
\]
\item $G = G^2 \circ G^1$, $\Sw = \Sw^2 \circ \Sw^1$, where for all $(\XX,\TT) \in \InH$, 
\begin{eqnarray*}
(G^2 \circ G^1)(\XX,\TT)&=&G^2(G^1(\XX,\TT),\Sw^1(\XX,\TT)), \\
(\Sw^2 \circ \Sw^1)(\XX,\TT)&=&\Sw^2(G^1(\XX,\TT),\Sw^1(\XX,\TT)).
\end{eqnarray*}
\end{enumerate}
\end{enumerate}
\end{cor}
\proof It is immediate from the definitions and constructions that if $\FAN^1, \FAN^2$ are regular FANs of simple type so
is $ \FAN^2 \conc \FAN^1 $. The remaining statements follow immediately. \qed

\begin{cor}
Suppose that $\FAN^1,\FAN^2,\dotsc,\FAN^q$ are regular FANs of simple type and that
$\FAN^j \prec \FAN^{j+1}$, $j = 1,\dotsc,q-1$. Then the concatenation
$\FAN^q \conc \FAN^{q-1} \conc \dotsc \conc \FAN^1$ is a well-defined regular FAN of simple type with
generalized transition function $G = G^q \circ \dotsc \circ G^1$, where $G^j$ is the generalization transition function of $\FAN^j$, $j \in \is{q}$.
\end{cor}
\proof We inductively define the concatenation $\FAN^s \conc \dotsc \conc \FAN^1$ to be
$\FAN^s \conc (\FAN^{s-1} \conc \dotsc \conc \FAN^1)$, $s > 2$. The result follows from proposition~\ref{concat} and corollary~\ref{cor1:conc}. \qed

\section{A structure theorem for regular FANs of simple type} 
\label{factorsec}

In this section, all FANs will be regular of simple type. This will be no loss of generality 
for our main application as, by theorem~\ref{thm:simp},  we can replace a regular FAN $\FAN$ by $\text{core}(\FAN)$.

\subsection{Primitive, stably primitive, trivial and basic FANs}

In section~\ref{join:sec} we defined trivial, indecomposable, stably indecomposable, and elementary FANs. The operation of concatenation allows us to
define an additional `irreducible' class of simple FAN.
\begin{Def}
Let $\FAN = (\Net,\In,\Tr)$ be a FAN with $k$ nodes.
\begin{enumerate}
\item $\NN$ is \emph{primitive} if $k> 1$ and
\begin{enumerate}
\item $\NN$ is indecomposable, and
\item $\NN$ cannot be written as a concatenation of two regular FANs.
\end{enumerate}
\item $\NN$ is \emph{stably primitive} if $k \ge 3$ and there
is a primitive FAN~$\is{P}$ and a trivial
FAN $\is{S} = \prod_{\ell\in\is{s}}\is{S}^\ell$ such that
$\NN = \is{P}\times \is{S}$.
\item $\FAN$ is \emph{basic} if $\NN=\bigsqcup_{a\in\is{p}}\FAN^a$,
where each $\FAN^a$ is stably primitive. Equivalently, if
$\FAN = \prod_{a \in \is{p}}\is{P}^a \times \prod_{\ell\in\is{q}}\is{S}^\ell$,
where the $\is{P}^a$ are primitive,  the $\is{S}^\ell$ are trivial, and $p,q\ge 0$.
\end{enumerate}
\end{Def}

\begin{prop}\label{isot}
Suppose that $\FAN^a = (\Net,\In^a,\Tr^a)$, $a\in \is{2}$ are basic FANs with common asynchronous network $\Net$.
Then there exists a continuous family $\{\FAN^t = (\Net,\In^t,\Tr^t) \dd t \in [1,2]\}$ of basic FANs connecting $\FAN^1$ to $\FAN^2$, and
smooth isotopies $J_i:M_i\times [1,2] \arr M_i$, $i \in \is{k}$,  such that for all $i\in\is{k}$,
\begin{enumerate}
\item $J_i(\In_i^1,1) = \In^1_i$, $J_i(\Tr_i^1,1)=\Tr^1_i$.
\item $J_i(\In_i^1,2) = \In^2_i$, $J_i(\Tr_i^1,2)=\Tr^2_i$.
\item The isotopy $J_i$ is the identity outside $(M_i^{1,0}\Delta M_i^{2,0})\cup V_i^1\cup V_i^2$ ($V_i^j$ is the open subset of $M_i$ used for the definition of
simple type, $j \in \is{2}$, and $M_i^{1,0}\Delta M_i^{2,0}$ is the symmetric difference).
\end{enumerate}
\end{prop}
\proof Fix $i \in \is{k}$. We use the vector field $Z_i = f_i^\ECS$ to define an 
isotopy~\cite{Hir} on $M_i$ that takes
$\In^1$ to $\In^2_i$ and $\Tr^1$ to $\Tr^2_i$. We can do this since the assumption of
simple type implies that the region `between' $\In_i^1$ and $\In^2_i$ is contained in $M_i^{1,0}\Delta M_i^{2,0}\subset \cup_{j \in \is{2}}(V_i^j \cup M_i^{j,-})$.
Similarly for the region between $\Tr_i^1$ and $\Tr^2_i$. 
It is straightforward to construct the isotopy so that it is the identity outside $M_i^{1,0}\Delta M_i^{2,0}\cup (V_i^1 \cup V_i^2)$.  \qed

\begin{rem}
In future, if we write $\FAN^1 \overset{\bullet}{=} \FAN^2$, we mean that $\Net^1 = \Net^2$ and that the initialization and termination sets for $\FAN^1$ and $\FAN^2$ are isotopic in the sense of proposition~\ref{isot}. The construction of  isotopies will always be along the lines given above 
and omitted.
\end{rem} 

Let $\FAN^1\overset{\bullet}{=}\FAN^2$ be FANs. 
We write $\FAN^1 \IS \FAN^2$ if $\In^1 = \In^2$, and
$\FAN^1 \TS \FAN^2$ if  $\Tr^1 = \Tr^2$.

\begin{prop}\label{struct1}
Let $\FAN$ be a regular FAN of simple type which is not basic. Suppose that there are
FANs $\is{R}^a,\is{Q}^a$, $a \in \is{2}$, such that $\FAN = \is{R}^1 \conc \is{Q}^1 = \is{R}^2 \conc \is{Q}^2$
and
\begin{enumerate}
\item $\is{Q}^a$ is stably primitive, $a \in \is{2}$.
\item $A(\is{Q}^1) \cap A(\is{Q}^2) \ne \emptyset$.
\end{enumerate}
Then we have
\begin{enumerate}
\item $\is{Q}^1\IS\is{Q}^2$  and $\is{R}^1 \TS \is{R}^2$.
\item $\is{Q}^1 \overset{\bullet}{=} \is{Q}^2$ and $\is{R}^1 \overset{\bullet}{=} \is{R}^2$
\end{enumerate}
In particular, if we can write $\FAN = \is{R}\conc \is{Q}$, where $\is{Q}$ is stably primitive, then
the decomposition is unique up to the choice of initialization and termination sets for $\is{R}, \is{Q}$. 
\end{prop}
\proof Since $\is{Q}^1$ and $\is{Q}^2$ are stably primitive, it follows that if
$A(\is{Q}^1)\cap  A(\is{Q}^2)\ne\emptyset$ then we must have $A(\is{Q}^1)= A(\is{Q}^2)$ -- there are no dependencies 
on trivial factors. It remains to prove that $\is{Q}^1 \TS  \is{Q}^2$. It is no loss of generality to assume $\is{k} = A(\is{Q}^1)= A(\is{Q}^2)$.
We construct the `intersection' FAN $\is{Q} = \is{Q}^1 \cap \is{Q}^2$. Specifically,
for $i \in \is{k}$, define
\[
M_i^{\is{Q},0} = M_i^{\is{Q}^1,0}\cap M_i^{\is{Q}^2,0},\quad M_i^{\is{Q},\sigma} = M_i^{\is{Q}^1,\sigma}\cup M_i^{\is{Q}^2,\sigma},\;\sigma\in \{+,-\}.
\]
Define the event map for $\is{Q}$ by restriction of the event maps for $\is{Q}^j$, $j \in \is{2}$.  Either 
$\is{Q}\TS\is{Q}^1,\is{Q}^2$ and we are done or not. If not, then we can write one of $\is{Q}^1$, $\is{Q}^2$ as a concatenation with $\is{Q}$. But stable primitivity implies that
if $\is{Q}^j = \is{T}^j \conc \is{Q}$, then $\is{T}^j$ is trivial, $j \in \is{2}$. \qed

\vspace*{0.1in}

If $\FAN = \is{R} \conc \is{Q}$ is the decomposition given by
proposition~\ref{struct1} and 
$ \is{Q} = \is{P} \times \is{S}$, 
where $\is{P}$ is the primitive factor of $\is{Q}$, then we take
\begin{enumerate}
\item $\In^{\is{R}}_i=\Tr^{\is{Q}}_i = \In_i$, if $i \in \is{k}\sm A(\is{Q})$,
\item $\In^{\is{R}}_i = \Tr^{\is{Q}}_i$, $i \in A(\is{Q})$.
\end{enumerate} 
\begin{prop} \label{fact}
Let $\FAN$ be a regular FAN of simple type which is not basic.  There exist FANs $\FAN^1$, $\is{R}^1$ such that
\begin{enumerate}
\item $\FAN^1$ is basic and $\FAN = \is{R}^1 \conc \FAN^1$.
\item $\FAN^1$ is maximal in the sense that if $\is{Q}$ is basic and $\FAN = \is{S} \conc \is{Q}$, then 
$A(\is{Q}) \subset A(\FAN^1)$ with equality iff $\is{Q}\overset{\bullet}{=} \FAN^1$.
\item  $\is{N}^1 \IS \FAN$ and
$\is{R}^1 \TS \FAN$.
\end{enumerate}
\end{prop}
\proof Repeated application of proposition~\ref{struct1}. \qed
\begin{rem}
We refer to $\FAN^1$ as a \emph{factor} of $\FAN$. Factors are always assumed to be basic.
\end{rem}
\begin{thm}[Factorization theorem]\label{fact1}
If $\FAN$ is a regular FAN of simple type, we may write
\[
\FAN = \FAN^q \conc \dotsc \conc \FAN^1,
\]
where
\begin{enumerate}
\item $\FAN^j$ is basic, $j \in \is{q}$.
\item The decomposition of $\FAN$ is unique, up to choice of initialization and termination sets, if we require that for $j=1,\dotsc,q-1$, $\FAN^j$ is the maximal
factor of  $(\FAN^q \conc \dotsc\conc \FAN^j)$. 
\end{enumerate}
\end{thm}
\proof The obvious induction, using proposition~\ref{fact}. Note that we only get finitely many factors on account of the compactness of 
trajectories joining $\In$ to $\Tr$ and the finite node set.\qed

\subsection{Conventions for labelling initialization and termination sets.}\label{label}
For $j \in \is{q}$, denote the initialization and termination sets for $\FAN^j$ by $\In^j$ and $\Tr^j$ respectively.
We always have $\In = \In^1$ and $\Tr = \Tr^q$. If $i \in \is{k} \sm A(\FAN^j)$, we take
$\Tr_i^j = \In_i^j = \Tr_i^{j-1}$, $j < q$ and $\Tr_i^q = \Tr_i$. If $i \in A(\FAN^j)$, then 
$\In_i^j = \Tr_i^{j-1} \ne \Tr_i^j$.

\begin{figure}[h]
\centering
\includegraphics[width=0.8\textwidth]{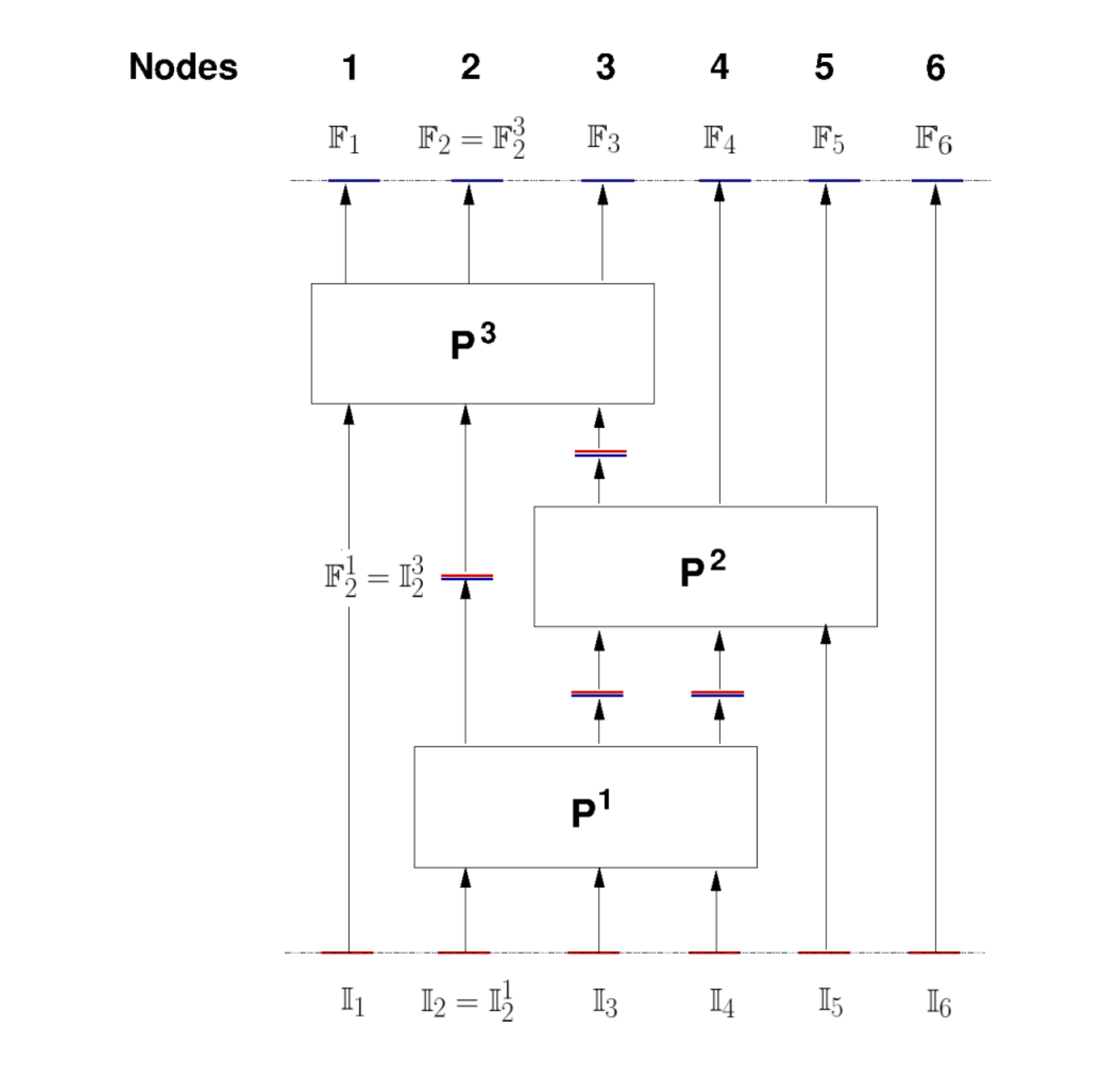}
\caption{A FAN with 6 nodes and 3 stably primitive factors.}
\label{spatio_temp0}
\end{figure}
\begin{exam}\label{order0}
Suppose $k = 6$ and $\FAN = \FAN^3\conc\FAN^2\conc\FAN^1$,
where each factor $\FAN^j$ is stably primitive and $A(\FAN^1) = \{2,3,4\}$, $A(\FAN^2) = \{3,4,5\}$, $A(\FAN^3) = \{1,2,3\}$. 
In figure~\ref{spatio_temp0}, we have shown initialization and termination sets for $N_2$ as well
indicated the new initialization and termination sets for the remaining nodes. 
For example, we have $\In_3^1 = \In_3$, $\Tr_3^1 = \In^2_3$, $\Tr_3^2 = \In_3^3$, $\Tr_3^3 = \Tr_3$.  
\examend
\end{exam}

\subsection{Partial order structure associated to a regular FAN}
Let $\FAN$ be a regular FAN which is nontrivial. Replacing $\FAN$ by $\text{core}(\FAN)$ we may and shall assume $\FAN$ is of simple type.
Let $\FAN = \FAN^q \conc \dotsc \conc \FAN^1$ be the decomposition of $\FAN$ given by theorem~\ref{fact1}. Since each
$\FAN^j$ is basic, we have the unique decomposition $\FAN^j = \bigsqcup_{\ell \in \is{p(j)}} \is{P}^{j,\ell}$, $j \in \is{q}$,
where each $\is{P}^{j,\ell}$ is stably primitive and $1 \le p(j) < k$, $j\in \is{q}$.
Necessarily $A(\is{P}^{j,\ell}) \cap A(\is{P}^{j,\ell'}) = \emptyset$ if $\ell \ne \ell'$.

For each $i \in \is{k}$, define 
\[
\is{q}(i) = \{j \in \is{q} \dd \exists \ell\in\is{p(j)} \; \text{such that } i\in A(P^{j,\ell})\}.
\]
Observe that for each $j \in \is{q}$, there is at most one $\ell\in \is{p(j)}$ such that $i\in A(P^{j,\ell })$. Let $\prec_i$
denote the restriction of the natural order $<$ on $\is{q}$ to $\is{q}(i)$.

Let $\mathfrak{E} = \{ \is{P}^{j,\ell } \dd j \in \is{q},\; \ell \in \is{p(j)}\}$ denote the set primitive events
comprising $\FAN$.  We define a partial order $\prec$ on $\mathfrak{E}$ generated by the relation
\[
\is{P}^{j,\ell } \prec \is{P}^{j+1,\ell'}\;\; \text{if there exists } i \in A(\is{P}^{j,\ell }) \cap A(\is{P}^{j+1,\ell'}). 
\]
Extend $\prec$ by transitivity to $\mathfrak{E}$. Observe that for all $j \in \is{q}$, $\is{P}^{j,\ell }, \is{P}^{j,\ell'}$ 
are not related except by equality if $\ell = \ell'$.  
\begin{lemma}
For each $i \in \is{k}$, the partial order $\prec$ on $\mathfrak{E}$ induces the total order $\prec_i$ on $\is{q}(i)$.
That is, if $a, b \in \is{q}(i)$, $a \prec_i b$, then there exists a (maximal) chain
$ \is{P}^{a,\ell} = \is{P}^{a_1,\ell_1} \prec \dotsc \prec \is{P}^{a_s,\ell_s }=\is{P}^{b,\ell_s }$, 
where $\is{q}(i) \cap [a,b] = \{a_1,\dotsc,a_s\}$ and $\ell_j \in \is{p(a_j)}$, $j \in \is{s}$.
\end{lemma}

\begin{exam}
With the notation of figure~\ref{spatio_temp0}, we have $q = 3$, $p(1)=p(2)=p(3) = 1$, 
$\is{P}^1 \prec\is{P}^2 \prec \is{P}^3$, $\is{q}(1) = \{3\}$, $\is{q}(2) = \{1,3\}$, $\is{q}(3) = \{1,2,3\}$, $\is{q}(4) = \{1,2\}$,
$\is{q}(5) = \{2\}$, and $\is{q}(6) = \emptyset$.
\end{exam}
\begin{prop}\label{ff}
(Notation and assumptions as above.)
The partial order $\prec$ on $\mathfrak{E}$ gives the associated network a natural feedforward structure.
\end{prop}
\begin{rem}
Proposition~\ref{ff} and theorem~\ref{fact1} together imply that every regular FAN determines a natural feedforward 
network. Regularity of the FAN implies there are no feedback loops between events. Of course, there may be
feedback loops within individual events.
\end{rem}
\begin{exam}\label{order}
\noindent We illustrate proposition~\ref{ff} with a more complex example. Referring to figure~\ref{event2x}, assume given a nine node
FAN $\FAN$ of simple type built from eight stably primitive FANs $\is{P}^a,\dotsc, \is{P}^h$. 
The decomposition given by theorem~\ref{fact1} is
\[
\FAN = \is{P}^h\conc (\is{P}^e\sqcup\is{P}^g)\conc(\is{P}^d \sqcup \is{P}^f)\conc\is{P}^b\conc(\is{P}^a \sqcup \is{P}^c)
\]
Note this decomposition is not unique amongst decompositions of minimal length $5$. For example, 
\[
\FAN = (\is{P}^h\sqcup\is{P}^g) \conc (\is{P}^c \sqcup\is{P}^e\sqcup\is{P}^f)\conc\is{P}^d \conc\is{P}^b\conc\is{P}^a 
\]
It is easy to show that all decompositions must have length at least 5, have the same partial order $\prec$, and induce the same total order on the $\is{q}(i)$, $i \in \is{9}$.

\begin{figure}[h]
\centering
\includegraphics[width=0.78\textwidth]{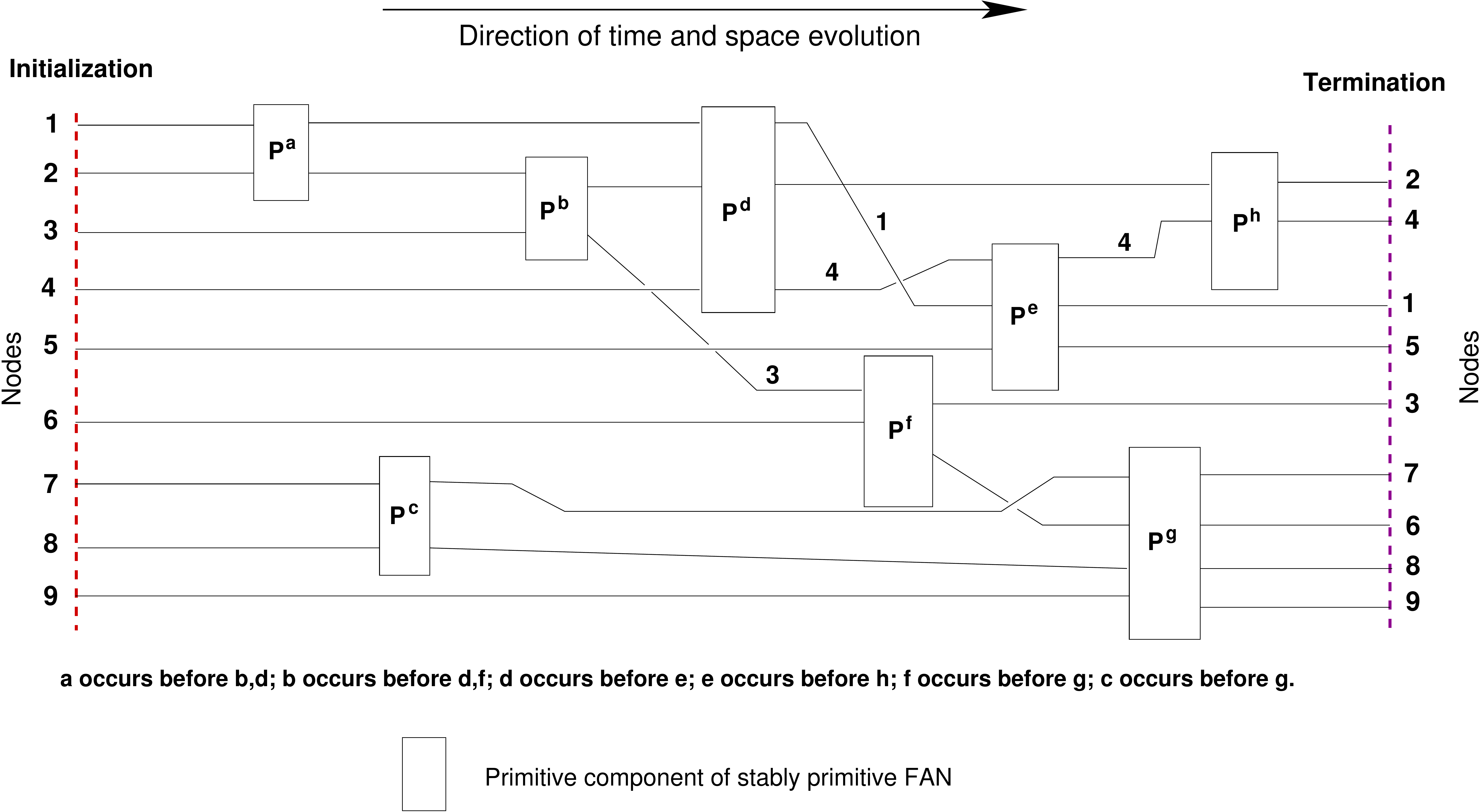}
\caption{Event sequence for a 9 node FAN.}
\label{event2x}
\end{figure}
\examend
\end{exam}
\subsection{Modularization of Dynamics Theorem}
\begin{thm}[Modularization of Dynamics]
\label{FAN}
Let $\FAN$ be a regular FAN of simple type and decomposition given by theorem~\ref{fact1} 
\[
\FAN = \FAN^q \conc \dotsc \conc \FAN^1,
\]
where $\FAN^j = \bigsqcup_{\ell \in \is{p(j)}} \is{P}^{j,\ell}$, $j \in \is{q}$,
and each $\is{P}^{j,\ell}$ is stably primitive, $j\in \is{q}$.
For simplicity, assume $\FAN$ has no trivial factors.

Let $G^j: \In^j \arr \Tr^j$ denote the generalized transition function for $\FAN^j$, $j \in \is{q}$ (we follow the convention of section~\ref{label}).
Then\[
G = G^q \circ \dotsc \circ G^1.
\]
Moreover, $G^j = \bigtimes_{\ell \in \is{p(j)}} G^{j,\ell}$, where the $G^{j,\ell }$ are the transition functions for the
primitive factors $\is{P}^{j,\ell}$ of $N^j$.
\end{thm}
\proof Immediate from our constructions and theorem~\ref{fact1}. \qed
\begin{rems}
(1) If we take a different decomposition of $\FAN$ -- as in example~\ref{order} -- we get a
different factorization of $G$. Basically, the order of composition is determined by 
the induced orders on $\is{q}(j)$, $j \in \is{q}$. \\
(2) If we allow $\FAN$ to have a trivial factor $\is{S}$ then the transition function for $\is{S}$ can be inserted anywhere in the 
decomposition $G = G^q \circ \dotsc \circ G^1$.
\end{rems}

We conclude with an example illustrating
theorem~\ref{FAN}.

\begin{exam}
We consider a FAN $\FAN= (\Net, \In, \Tr)$ that models three trains
$\mathfrak{T}_1, \mathfrak{T}_2, \mathfrak{T}_3$ passing through two
passing loops, see figure~\ref{fig:TwoPassingLoops}. The state
$x_i$ of $\mathfrak{T}_i$ is
given by its position on the real line $M_i=\real$, $i \in \is{3}$.  The passing
loops are located at $0, L$ and stations $A,B,C$ are at $p, -q, r$ respectively, where
$L, p, q, r>0$. The trains~$\mathfrak{T}_1, \mathfrak{T}_3$ start at stations $A$ and $B$
and travel with velocities
$v_1< v_3 < 0$ respectively. The train~$\mathfrak{T}_2$ travels
with velocity $v_2>0$ and starts at station $C$. It is required that
$\mathfrak{T}_1$ has to go through the passing loop at $0$ to pass train~$\mathfrak{T}_2$, and
that train~$\mathfrak{T}_2$ then has to traverse the passing loop at $L$ to pass
train~$\mathfrak{T}_3$.  The trains $\mathfrak{T}_1, \mathfrak{T}_3$ terminate at
$C$, and $\mathfrak{T}_2$ terminates at $B$.

\begin{figure}
\centering
\includegraphics[width=0.95\textwidth]{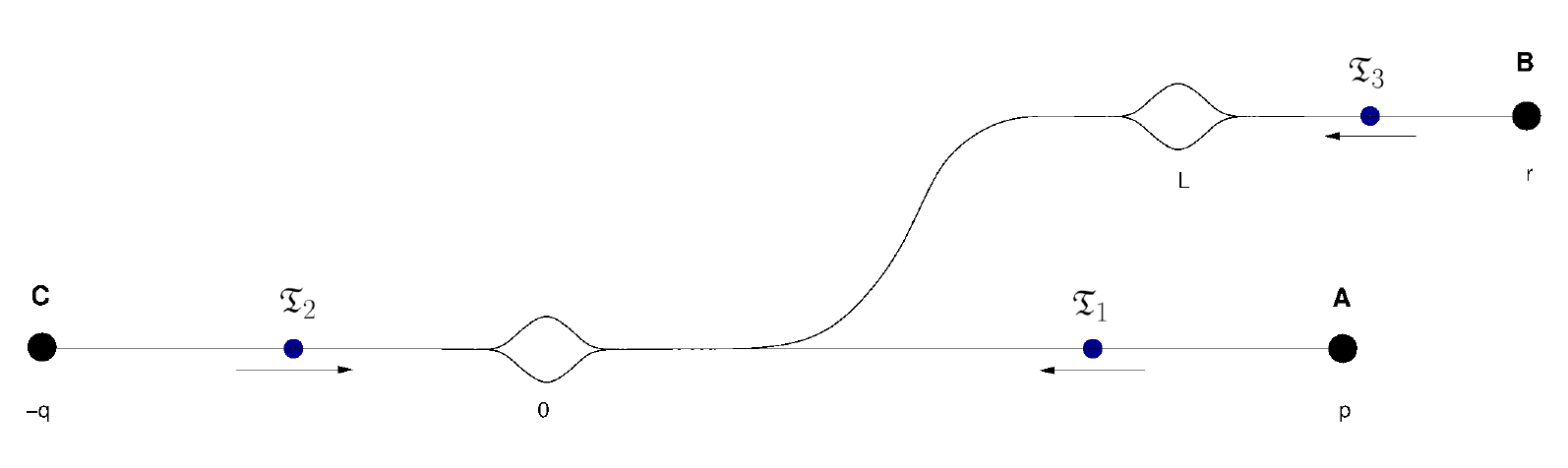}
\caption{Three trains going through two passing loops. We require
that once train~$\mathfrak{T}_2$ has traversed the first passing loop
it will continue on the branch line towards the second passing
loop.}
\label{fig:TwoPassingLoops}
\end{figure}

We take node set $\Nn=\sset{\NS, N_1, N_2, N_3}$
and network phase space $\Mb=\R^3$. Let $\alpha_i=N_0\to N_i$, $i \in \is{3}$.
We define the generalized connection structure
$\A = \sset{\ECS, \alpha_1, \alpha_2,\alpha_3, \alpha_1\vee\alpha_3, \alpha_2\vee\alpha_3}$.
We define the $\A$-structure $\F$ by
\begin{align*}
\f^\ECS&=(v_1, v_2, v_3), & \f^{\alpha_1}&=(0, v_2, v_3),\\
\f^{\alpha_2}&=(v_1, 0, v_3),& \f^{\alpha_3}&=(v_1, v_2, 0),\\
\f^{\alpha_1\vee\alpha_3}&=(0, v_2, 0),& 
\f^{\alpha_2\vee\alpha_3}&=(v_1, 0, 0).
\end{align*}
We define the event map $\Ee:\Mb\arr \A$ by
\[
\Ee(x_1, x_2, x_3) = \begin{cases}
\alpha_1 & \text{if } x_1=0, x_2<0,\\
\alpha_2 & \text{if } x_1>0, x_2=0 \text{ or } x_2=L, x_3>L,\\
\alpha_3 & \text{if } x_2<L, x_3=L,\\
\alpha_1\vee\alpha_3 & \text{if } x_1=0, x_2<0, x_3=L,\\
\alpha_2\vee\alpha_3 & \text{if } x_1<0, x_2=0, x_3=L,\\
\ECS & \text{otherwise}.
\end{cases}
\]
These definitions define the asynchronous network $\Net = (\Nn, \A, \F, \Ee)$.
We obtain the FAN $\FAN= (\Net, \In, \Tr)$ modelling the train network by taking
initialization and termination sets $\In, \Tr$ defined by
$\In_1 = \sset{p}$, $\Tr_1 = \sset{-q}$,
$\In_2 = \sset{-q}$, $\Tr_2 = \sset{r}$,
$\In_3 = \sset{r}$, $\Tr_3 = \sset{-q}$.

We identify two
stably primitive components, $\FAN^a$ (describing dynamics in
the first passing loop), and~$\FAN^b$ (describing dynamics in the second passing loop).
We define~$\Net^a$ by
$\A^a = \sset{\ECS, \alpha_1, \alpha_2}$, $\F^a = \tsset{\f_a^{\alpha_1} = \f^{\alpha_1}, \f_a^{\alpha_2} = \f^{\alpha_2},\f_a^\ECS = \f^\ECS}$, and take
\[
\Ee^a(x_1, x_2, x_3) = \begin{cases}
\alpha_1 & \text{if } x_1=0, x_2<0,\\
\alpha_2 & \text{if } x_1>0, x_2=0\\
\ECS & \text{otherwise}.
\end{cases}
\]
For $\Net^b$ we take
$\A^b = \sset{\ECS, \alpha_2, \alpha_3}$, $\F^b = \tsset{\f_b^{\alpha_2} = \f^{\alpha_2}, \f_b^{\alpha_3} = \f^{\alpha_3},\f_b^\ECS = \f^\ECS}$,
\[
\Ee^b(x_1, x_2, x_3) = \begin{cases}
\alpha_2 & \text{if } x_2=L, x_3>L,\\
\alpha_3 & \text{if } x_2<L, x_3=L,\\
\ECS & \text{otherwise}.
\end{cases}
\]
We have $\Net = \Net^a \conc \Net^b$. 
Next we define the
initialization and termination sets for $\Net^a,\Net^b$.
We take $\In^c_i = \In_i$, $c \in \{a,b\}$, $i \in \is{3}$, except that
$\In_2^b = \{L/2\}$ (any point in $(0,L)$ would do. Similarly,
we take $\Tr^c_i = \Tr_i$, $c \in \{a,b\}$, $i \in \is{3}$, except that
$\Tr_2^a = \{L/2\}$. With these definitions, $\FAN^a = (\Net^a,\In^a,\Tr^a)$ and
$\FAN^b = (\Net^b,\In^b,\Tr^b)$ are stably primitive FANs and $\Net = \Net^a \conc \Net^b$,
If we denote the primitive factors of $\FAN^a $, $\FAN^b$ by $\is{P}^a$, $\is{P}^b$
respectively, then $A(\is{P}^a) = \sset{1,2}$, and $A(\is{P}^b) = \sset{2,3}$.

Applying theorem~\ref{FAN}, the generalized transition function
for~$\FAN$ can be written as a composition of the generalized transition
functions $G^a$ for $\FAN^a$ and $G^b$ for $\FAN^b$. Note that~$G^a$ is the identity
in the third component, $G^b$ is the identity in the first component.
Different initializations
yield different trajectories as depicted in
figure~\ref{fig:TwoPassingLoopsSchem}(b).
\begin{figure}
\centering
\includegraphics[width=\textwidth]{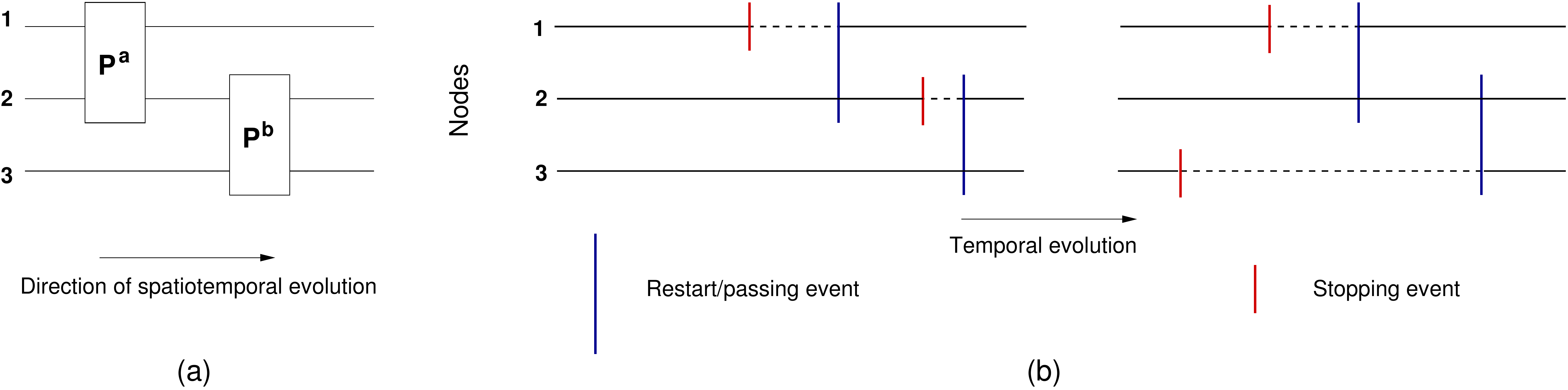}
\caption{(a) Schematic indicating the primitive components of $\FAN$.
(b) Temporal evolution for individual nodes based on
different starting times. Dotted lines indicate stopped nodes, vertical
bars either stopping or restarting/passing events. The stopping
events may occur
in any order in contrast to the restarting events.}
\label{fig:TwoPassingLoopsSchem}
\end{figure}
\examend
\end{exam}

\section{Concluding comments \& outstanding problems}
\label{conclu}
Our overall aim has been to outline a mathematical framework for 
asynchronous networks and event driven dynamics that enables the 
analysis of network dynamics that cannot be satisfactorily 
modelled by classical systems of analytic differential equations. 
Examples and applications are motivated by problems in 
computer science, biology and engineering. In particular, there 
are connections with Filippov systems (mostly considered in the
engineering literature).
As we indicated in the introduction, functional networks,
have previously been 
considered in neuroscience. 

The modularization of dynamics theorem shows that it is possible to utilize a
reductionist approach in nonlinear systems -- and so give an answer to Alon's question, posed in the introduction of~\cite{BF0} -- provided that
one emphasizes \emph{function} rather than \emph{dynamics}. 
Our approach does not work if we attempt to approximate an asynchronous network by a synchronous 
network using, for example, averaging or thermodynamic formalism. In other words, in order to 
understand the global dynamics of asynchronous networks, one  needs to work with, rather than against, the
inherent nonsmoothness and take account of network function.

Many questions and open problems remain. We sketch a few 
representative mathematical questions mainly about functional 
asynchronous networks.
 
Much of the theory developed in sections~\ref{join:sec}, \ref{factorsec} continues to apply if we work with weakly regular 
networks of simple type. However, the modularization of dynamics theorem no longer applies. Can the definition
of weak regularity be modified so as to get a useful version of modularization of dynamics in the weakly regular case? In this regard,
it would be helpful to get a better description and classification 
of hidden deadlocks, especially those arising through bifurcation of a regular network. 

There are several natural questions concerning bifurcation. First, if we assume weak regularity, what types of bifurcation can occur as
we start to vary starting times? This is already interesting in the simplest case of a 
primitive FAN which is only weakly regular. Secondly, from an evolutionary perspective, 
one might expect that regular FANs of simple type occurred early in evolutionary development. 
What types of bifurcation can occur in the process of optimization of function -- for example by adding
feedback loops between events in the feedforward structure of a FAN given by theorem~\ref{fact1}?
It would also be useful to 
have  a structural classification of primitive FANs with a small 
number of nodes and an identification of the primitive FANs 
which appear most frequently in applications.

For some applications, it may be appropriate to replace the termination hypersurface used in the definition of a FAN by a proper closed subset 
that is reached for a particular time initialization, say 
$\is{0}=(0,\dotsc,0) \in \real^k_+$. There is the question of how the target 
sets may vary and bifurcate as we increase the range of possible 
initialization times. 
This question is of direct relevance to applications: 
initialization at $\is{0}$ can be seen as `synchronized 
initialization' and network function may break down if the 
initialization times are too spread out. If we have a FAN with a
generalized transition function that spreads the termination times 
out on average (and termination times yield initialization times 
for another FAN) then after a certain amount of repetition network 
function may break down. In terms of a transportation network this 
could be seen as propagation of delays. Real-world 
transportation networks networks are typically approximately 
synchronized on a daily basis through a nightly `reset'.
More generally, for realistic applications it is usually natural 
to assume the initialization times, and other starting time events, 
follow a statistical law and obtain the corresponding statistical 
law of the termination times. 

Finally, using modularization, we anticipate that further insights into the dynamics of real world networks can be made. 
A crucial point is to understand the primitive factors and feedforward structure of the underlying FAN based on real-world time series data. 
One possible approach to determine individual modules could be to use dynamic Bayesian inference to infer how 
connections in the network change over time; see for example~\cite{STM}. 
Of particular interest would be to find the original structures in an ``evolved'' 
functional network that can no longer be decomposed into simple primitive factors on account of feedback loops evolving between the 
original primitive events. 
Moreover, modularization relates to network design and evolution and it is natural to attempt to
find or design the optimal asynchronous network to perform a desired network function. 
While such questions have been discussed within the context of control, for example~\cite{Borgers2014, Adaldo2015}, network 
analysis based on modularization of dynamics allows us to tackle these questions in a much wider context.

\end{document}